\documentclass[a4paper,11pt,reqno]{amsart}

\usepackage{amsmath,amsthm}
\usepackage{amssymb,latexsym}
\usepackage{amsaddr}
\usepackage{graphicx,subfig}
\usepackage{epsfig}
\usepackage{color}

\newtheorem{thm}{Theorem}[section]

\newtheorem{lma}[thm]{Lemma}

\newtheorem{prp}[thm]{Proposition}

\newtheorem{clm}[thm]{Claim}

\newtheorem{conj}[thm]{Conjecture}

\numberwithin{equation}{section}

\def\eps{\varepsilon}
\def\Dmon{\Delta_{{\rm mon}}}

\title{An edge-coloured version of Dirac's Theorem}

\author{Allan Lo}
\address{School of Mathematics, University of Birmingham,\\Birmingham, B15 2TT, UK}
\email{s.a.lo@bham.ac.uk}
\thanks {The research leading to these results was supported by the  European Research Council
under the ERC Grant Agreement no. 258345.
}
\date{\today}
\keywords{proper edge-coloring, $2$-factor, Hamiltonian cycle}
\begin{document}

\begin{abstract}
Let $G$ be an edge-coloured graph. 
The minimum colour degree $\delta^c(G)$ of $G$ is the largest integer $k$ such that, for every vertex~$v$, there are at least $k$ distinct colours on edges incident to~$v$.
We say that $G$ is properly coloured if no two adjacent edges have the same colour.
In this paper, we show that every edge-coloured graph $G$ with $\delta^c(G) \ge  2|G|/3$ contains a properly coloured $2$-factor.
Furthermore, we show that for any $\eps > 0 $ there exists an integer $n_0$ such that every edge-coloured graph $G$ with $|G| = n \ge n_0$ and $\delta^c(G) \ge ( 2/3 + \eps ) n $ contains a properly coloured cycle of length $\ell$ for every $3 \le \ell \le n$.
This result is best possible in the sense that the statement is false for $\delta^c(G) < 2n/3 $.
\end{abstract}

\maketitle

\section{Introduction}

A classical theorem of Dirac~\cite{MR0047308} states that every graph $G$ on $n \ge 3$ vertices with minimum degree $\delta(G) \ge n/2$ contains a Hamiltonian cycle.
In this paper, we generalise this result to edge-coloured graphs.

An \emph{edge-coloured graph} is a graph $G$ with an edge colouring~$c$ of~$G$.
We say that $G$ is \emph{properly coloured} if no two adjacent edges of $G$ have the same colour.
Moreover, $G$ is said to be \emph{rainbow} if all edges have distinct colours. 
We consider the following analogue of degree for edge-coloured graphs~$G$.
Given a vertex $v \in V(G)$, the \emph{colour degree $d^c(v)$} is the number of distinct colours of edges incident to~$v$.
The \textit{minimum colour degree} $\delta^c(G)$ of an edge-coloured graph~$G$ is the minimum $d^c(v)$ over all vertices~$v$ in~$G$.

One elementary result of graph theory states that every graph $G$ with $\delta(G) \ge 2$ contains a cycle.
However, for all $k \ge 2$, there exist edge-coloured graphs $G$ with $\delta^c(G) \ge k$ that do not contain any properly coloured cycles. 
Grossman and H{\"a}ggkvist~\cite{MR701173} gave a sufficient condition for the existence of properly coloured cycles in edge-coloured graphs with two colours, which does not depend on $\delta^c(G)$.
Later on, Yeo~\cite{MR1438622} extended the result to edge-coloured graphs with any number of colours.

One natural generalisation of Dirac's theorem is to determine the minimum colour degree threshold for the existence of a rainbow Hamiltonian cycle.
However, such thresholds do not exist for all $n$.
Indeed, for every even integer $n >3$, there exists a properly coloured complete graph $K_n^c$ on $n$ vertices using exactly $n-1$ colours.
Note that $\delta^c(K_n^c) = n-1$, but $K_n^c$ does not contain a rainbow Hamiltonian cycle.
In fact, for each $p \ge 2$, there is a properly coloured~$K_{2^p}^c$ that does not contain any rainbow Hamiltonian path (see~\cite{MR758880}).
Hence, a better question would be to ask for the minimum colour degree threshold for the existence of a properly coloured Hamiltonian cycle.

The problem of finding properly coloured spanning subgraphs in edge-coloured complete graphs $K_n^c$ has been investigated by numerous researchers.
Bang-Jensen, Gutin and Yeo~\cite{MR1609957} proved that if $K_n^c$ contains a properly coloured $2$-factor, then $K_n^c$ also contains a properly coloured Hamiltonian path.
A graph $G$ is said to be a \emph{$1$-path-cycle} if $G$ is a vertex-disjoint union of at most one path and a number of cycles.
Feng, Giesen, Guo, Gutin, Jensen and Rafiey~\cite{MR2270727} showed that $K_n^c$ contains a properly coloured Hamiltonian path if and only if it contains a spanning properly coloured $1$-path-cycle.
We define \emph{$\Dmon(K_n^c)$} to be the maximum number of edges of the same colour incident to the same vertex.
In other words, $\Dmon (K_n^c) = \max \Delta(H)$ over all monochromatic subgraphs $H \subseteq K_n^c$, where an edge-coloured graph $G$ is \emph{monochromatic} if all edges have the same colour.
Notice that $\Dmon (K_n^c) + \delta^c (K_n^c) \le n$.
Bollob\'as and Erd\H{o}s~\cite{MR0411999} proved that if $\Dmon(K_n^c) \le n/69$, then $K_n^c$ contains a properly coloured Hamiltonian cycle.
They further conjectured that $\Dmon(K_n^c)  < \lfloor n/2\rfloor$ suffices.
Their result was subsequently improved by Chen and Daykin~\cite{MR0422070}, Shearer~\cite{MR523092} and Alon and Gutin~\cite{MR1610269}.
In~\cite{LoPCHCinKn}, the author showed that for any $\eps >0$, every $K_n^c$ with $\Dmon(K_n^c)  < (1/2 - \eps) n $ contains a properly coloured Hamiltonian cycle, provided $n$ is large enough.
Therefore, the conjecture of Bollob\'as and Erd\H{o}s is true asymptotically.
For a survey regarding properly coloured subgraphs in edge-coloured graphs, we recommend Chapter~16 of~\cite{MR2472389}.

Let $G$ be an edge-coloured graph (not necessarily complete).
Li and Wang~\cite{MR2519172} proved that $G$ contains a properly coloured path of length $2\delta^c(G)$ or a properly coloured cycle of length at least $2\delta^c(G)/3$.
In~\cite{Lo10}, the author improved this result by showing that $G$ contains a properly coloured path of length $2\delta^c(G)$ or a properly coloured cycle of length at least $\delta^c(G)+1$.
In the same paper, the author showed that every connected edge-coloured graph~$G$ contains a properly coloured Hamiltonian cycle or a properly coloured path of length at least $6\delta^c(G) /5-1$.
Furthermore, the author also conjectured the following. 

\begin{conj}[\cite{Lo10}] \label{conj:path}
Every connected edge-coloured graph $G$ contains a properly coloured Hamiltonian cycle or a properly coloured path of length $\lfloor 3\delta^c(G)/2 \rfloor$.
\end{conj}

If the conjecture is true, then every edge-coloured graph $G$ with $\delta^c(G) \ge 2 |G|/3$ contains a properly coloured Hamiltonian cycle. 
In this paper, we prove the following weaker result that $\delta^c(G) \ge 2 |G|/3$ implies the existence of a properly coloured $2$-factor in~$G$. 

\begin{thm} \label{thm:2factor}
Every edge-coloured graph $G$ with $\delta^c(G) \ge 2|G|/3$ contains a properly coloured $2$-factor.
\end{thm}

We also show that if $\delta^c(G) \ge (2 /3 + \eps)|G|$ and $|G|$ is large enough, then $G$ does indeed contain a properly coloured Hamiltonian cycle. 
Furthermore, $G$ contains a properly coloured cycle of length $\ell$ for all $3 \le \ell \le |G|$.

\begin{thm} \label{thm:PCHC}
For any $\eps >0$, there exists an integer $n_0$ such that every edge-coloured graph $G$ with $\delta^c(G) \ge (2/3+ \eps) |G| $ and $|G| \ge n_0$ contains a properly coloured cycle of length $\ell$ for all $3 \le \ell \le |G|$.
\end{thm}

Given a subgraph $H$ in $G$, we write $G -H$ for the (edge-coloured) subgraph obtained from $G$ by deleting all edges in~$H$.
For an edge-coloured graph~$G$, let \emph{$\delta_1^c(G)$} be the minimum $\delta(G - H)$ over all monochromatic subgraphs~$H$ in~$G$.
Note that $\delta^c_1(G) \ge \delta^c(G) - 1$.
Note that for an edge-coloured complete graph $K_n^c$, we have $\delta^c_1(K_n^c) + \Dmon(K_n^c) = n-1$.
We prove the following stronger statement, which implies Theorem~\ref{thm:PCHC}.

\begin{thm} \label{thm:PCHC2}
For any $\eps >0$, there exists an integer $n_0$ such that every edge-coloured graph $G$ with $\delta^c_1(G) \ge (2/3+ \eps) |G| $ and $|G| \ge n_0$ contains a properly coloured cycle of length $\ell$ for all $3 \le \ell \le |G|$.
\end{thm}

We now outline the proof of Theorem~\ref{thm:PCHC2} for the case when $\ell = |G|$, i.e. the existence of a properly coloured Hamiltonian cycle.
The proof adapts the absorption technique introduced by R\"odl, Ruci\'{n}ski and Szemer\'{e}di~\cite{MR2399020}, which was used to tackle Hamiltonicity problems in hypergraphs.
The proof is divided into two main steps.
In the first step, we find (by Lemma~\ref{lma:abscycle}) a small `absorbing cycle' $C$ in $G$.
The absorbing cycle $C$ has the property that, given any small number of vertex-disjoint properly coloured paths $P_1, P_2, \dots, P_k$ in $G$ with $V(C) \cap V(P_i) = \emptyset$ for each $i \le k$, there exists a properly coloured cycle $C'$ in $G$ with $V(C') = V(C) \cup \bigcup_{1 \le i \le k} V(P_i)$.
Thus, we have reduced the problem to covering the vertex set $V(G) \setminus V(C)$ with small number of vertex-disjoint properly coloured paths.
We remove the vertices of $C$ from $G$ and let $G'$ be the resulting graph.
Since $C$ is small, we may assume that $\delta^c_1(G') \ge (2/3 + \eps') |G'|$ for some small $\eps' >0$.
Then we find a properly coloured $2$-factor in $G'$ using Lemma~\ref{lma:2factor1} such that every cycle has length at least $\eps' |G'|/2$.
(Although Theorem~\ref{thm:2factor} also implies that $G'$ contains a properly coloured $2$-factor, there is no bound on the lengths of the cycles.)
Hence $V(G')$ can be covered by at most $2 / \eps'$ vertex-disjoint properly coloured paths $P_1, P_2, \dots, P_k$.
By the `absorbing' property of $C$, there is a properly coloured cycle $C'$ with $V(C') = V(C) \cup \bigcup_{1 \le i \le k} V(P_i) = V(G)$.
Therefore, $C'$ is a properly coloured Hamiltonian cycle as required.

The paper is organised as follows.
In the next section, we set up some basic notation and give the extremal example for Conjecture~\ref{conj:path}.
Section~\ref{sec:2fact} is devoted to finding properly coloured $2$-factors and contains the proof of Theorem~\ref{thm:2factor}.
The absorbing cycle is constructed in Section~\ref{sec:abscycle}.
Finally, Theorem~\ref{thm:PCHC2} is proved in Section~\ref{sec:proof}.


\section{Notation and extremal example}

Throughout this paper, unless stated otherwise, $G$ will be assumed to be an edge-coloured graph with edge-colouring~$c$. 
For an edge $xy$, we denote its colour by $c(xy)$.
For $v \in V(G)$, we denote by $N_G(v)$ the neighbourhood of~$v$ in $G$.
If the graph $G$ is clear from the context, we omit the subscript.
Let $U \subseteq V(G)$.
We write $G[U]$ for the (edge-coloured) subgraph of $G$ induced by~$U$.
We also write $G \setminus U$ for the subgraph obtained from $G$ by deleting all vertices in~$U$, i.e. $G \setminus U = G[V(G) \setminus U]$.
For edge-disjoint (edge-coloured) graphs $G$ and $H$, we denote by $G + H$ the union of $G$ and~$H$.
Further, we write $G - H +H'$ to mean $(G - H ) + H'$.

Let $U,W \subseteq V(G)$ not necessarily disjoint.
Whenever we define an auxiliary bipartite graph $H$ with vertex classes $U$ and~$W$, we mean $H$ has vertex classes $U'$ and $W'$, where $U'$ is a copy of~$U$ and $W'$ is a copy of~$W$.
Hence, $U$ and $W$ are considered to be disjoint in~$H$.
Given an edge $uw$ in~$H$, we say $u \in U$ and $w \in W$ to mean $u \in U'$ and $w \in W'$.

In this paper, every path will be assumed to be directed.
Hence, the paths $v_1 v_2 \dots v_{\ell}$ and $v_{\ell} v_{\ell-1} \dots v_1$ are considered to be different for $\ell \ge 2$.
We also allow a single vertex to be a (trivial) path.
Given vertex-disjoint paths $P_1, \dots,P_s$, we define the path $P_1 \dots P_s$ to be the concatenation of $P_1, \dots, P_s$.
For example, if $P= v_1  \dots v_{\ell}$ and $Q = w_1 \dots w_{\ell'}$ are vertex-disjoint paths and $x$ is a vertex not in $V(P) \cup V(Q)$, then $P xQ$ denotes the path $v_1 \dots v_{\ell} x w_1 \dots w_{\ell'}$.

Let $H$ be a union of vertex-disjoint directed cycles.
For each $y \in V(H)$, let $C_y$ be the cycle in~$H$ that contains~$y$.
Denote by $y_+$ and $y_-$ the successor and ancestor of $y$ in $C_y$ respectively.
Write $y_{--}$ for $(y_-)_-$ and $y_{++}$ for $(y_+)_+$. 
For distinct vertices $y,z$ in a cycle $C$ in $H$, define $yC^+ z$ and $y C^- z$ to be the paths $y y_+ \dots z_- z$ and $y y_- \dots z_+ z$ in $C$ respectively.
If $y=z$, then set $yC^+y = y = yC^-y$.

In~\cite{Lo10}, the author gave a construction to show that Conjecture~\ref{conj:path} is best possible for infinitely many values of $|G|$ and $\delta^c(G)$.
The same construction also shows that Theorem~\ref{thm:2factor} is best possible.
We include it here for completeness.

\begin{prp}
For $n,\delta \in \mathbb{N}$ with $\delta < 2n/3$, there exists a connected edge-coloured graph $G$ on $n$ vertices with $\delta^c(G) = \delta$ which does not contain a union of vertex-disjoint properly coloured cycles spanning more than $3 \delta^c(G)/2$ vertices.
In particular, $G$ does not contain a properly coloured $2$-factor or a properly coloured Hamiltonian cycle.
Moreover, all properly coloured paths in $G$ have length at most $3 \delta^c(G)/2$.
\end{prp}

\begin{proof}
Let $n,\delta \in \mathbb{N}$ with $3 \delta < 2n $.
Let $X = \{x_1, x_2,\dots, x_{\delta} \}$ and $Y$ be vertex sets with $|Y| = n- \delta$ and $X \cap Y = \emptyset$.
Define $G$ to be the edge-coloured graph on the vertex set $X \cup Y$ as follows.
Let $G[X]$ be a rainbow complete graph and let $G[Y]$ be a set of independent vertices.
For each $1 \le i \le \delta$, add an edge of new colour~$c_i$ between $x_i$ and $y$ for each $y \in Y$.
By our construction, $d^c(v) = |X| = \delta$ for all $v \in V(G)$ and so $\delta^c(G) = \delta$.

Let $C$ be a properly coloured cycle in $G$.
Arbitrarily orient $C$ into a directed cycle. 
Note that every vertex $y \in V(C) \cap Y$ must be immediately followed by two consecutive vertices in $X$, so $|X \cap V(C)| \ge 2 |Y \cap V(C)|$.
Therefore, if $\mathcal{C}$ is a collection of vertex-disjoint properly coloured cycles in $G$, then $\mathcal{C}$ spans at most $|X| + |X|/2  = 3 \delta / 2 $ vertices.
The `moreover' statement is proved by a similar argument.
\end{proof}


\section{Properly coloured $2$-factors} \label{sec:2fact}

First we prove Theorem~\ref{thm:2factor}; that is, every edge-coloured graph with $\delta^c(G) \ge 2|G|/3$ contains a properly coloured $2$-factor.
We now present a sketch of the proof.
Suppose that $G$ is an edge-coloured graph on $n$ vertices.
Let $H$ be a properly coloured subgraph in $G$ consisting of vertex-disjoint cycles such that $|H|$ is maximal.
If $V(H) = V(G)$, then we are done.
Hence we may assume that there exists $x \in V(G) \setminus V(H)$.
Define a \emph{colour neighbourhood} $N^c(x)$ of $x$ in $G$ to be a maximal subset of neighbours of $x$ in $G$ such that $c(x y) \ne c(x z)$ for all distinct $y,z \in N^c(x)$.
(The choice of $N^c(x)$ will be specified later.)
Assume that we are in the ideal case that $N^c(x) \cap V(H) = \emptyset$.
If $yz$ is an edge and $c(xy) \ne c(yz) \ne c(xz)$ for some distinct $y,z \in N^c(x)$, then $xyzx$ is a properly coloured triangle and so $H + xyzx$ contradicts the maximality of $|H|$. 
Hence, we may assume that for every $y,z \in N^c(x)$, $yz$ is not an edge, or $c(zy) = c(xy)$, or $c(zy) = c(xz)$.
By an averaging argument, there exists a vertex $y_0 \in N^c(x)$ such that the number of $z \in N^c(x) \setminus \{ y_0 \}$ such that either $y_0 z \notin E(G)$ or $c(y_0 z) = c(x y_0)$ is at least $(|N^c(x)|-1)/2 \ge ( \delta^c(G) - 1 ) / 2$.
Recall that there are at least $\delta^c(G)-1$ neighbours $v$ of $y_0$ with $c(y_0 v) \ne c(x y_0)$.
This means that $n = |V(G)| \ge (\delta^c(G) -1 )/2 + (\delta^c(G) -1 ) + |\{x, y_0\}| > 3 \delta^c(G)/2$ and so $\delta^c(G) < 2n/3$.

\begin{proof}[Proof of Theorem~$\ref{thm:2factor}$]
Let $G$ be an edge-coloured graph on $n$ vertices with edge-colouring $c$.
Set $\delta = \delta^c(G)$.
Suppose that $G$ does not contain a properly coloured $2$-factor.
We will show that $\delta < 2n/3$.
This is trivial if $n \le 2$, so we may assume that $n \ge 3$.
By deleting edges in $G$, we may assume that $G$ is edge-minimal, that is, any additional edge deletion would lead to a decrease in $d^c(v)$ for some vertex $v$ in~$G$.
Suppose that $G'$ is a monochromatic subgraph of $G$, which is not isomorphic to a disjoint union of stars.
Then there exists an edge $uv$ in $G'$ with $d_{G'}(u), d_{G'}(v) \ge 2$.
Set $G'' = G - uv$.
Note that $d^c_{G''}(x) = d^c_{G}(x)$ for all vertices $x \in V(G)$.
This contradicts the edge-minimality of~$G$. 
Therefore, every monochromatic subgraph $G'$ in $G$ is a disjoint union of stars.
Let $H$ be a properly coloured subgraph in $G$ consisting of vertex-disjoint cycles such that $|H|$ is maximal.
Note that $|H| < n$.
Arbitrarily orient each cycle in~$H$ into a directed cycle.

Fix $x \in V(G) \setminus V(H)$.
We will choose $N^c(x)$ to be a maximal subset of neighbours of $x$ in $G$ such that $c(x y_1) \ne c(x y_2)$ for all distinct $y_1, y_2 \in N^c(x)$.
Note that $|N^c(x)| = d^c(x)$.
Obtain $N^c(x)$ as follows.
If there are at least two vertices $y_1$ and $y_2$ such that $c( x y_1) =c( x y_2 )$, then we choose $y \in \{ y_1, y_2 \}$ to be in $N^c(x)$ according to the following order of preferences (if there are still at least two choices for $y$, pick one arbitrarily):
\begin{itemize}
	\item[(a)] $y \notin V(H)$;
	\item[(b)] $y \in V(H)$ and $c(xy) = c(yy_+)$;
	\item[(c)] $y \in V(H)$ and $c(xy) = c(yy_-)$;
	\item[(d)] $y \in V(H)$ and $xy_+ \notin E(G)$;
	\item[(e)] $y \in V(H)$ and $c(xy) \ne c(xy_+)$;
	\item[(f)] $y \in V(H)$ and $c(xy) = c(xy_+)$.
\end{itemize}
Note that if $y \in N^c(x)$ satisfies property~(f), then $c(xy') = c(xy)$ for every vertex $y'$ in $C_y$.
Define $N^c_{{\rm (a)}}(x)$ to be the set of vertices in $N^c(x)$ chosen due to~(a), and define $N^c_{{\rm (b)}}(x), \dots, N^c_{{\rm (f)}}(x)$ similarly.
Next, set
\begin{align*}
	W & = N^c_{{\rm (a)}}(x), \\
	R' & = \{y_+ : y \in N^c_{{\rm (b)}}(x) \cup N^c_{{\rm (d)}}(x) \cup N^c_{{\rm (e)}}(x) \cup  N^c_{{\rm (f)}}(x) \},\\
	S' & = \{y_- : y \in N^c_{{\rm (c)}}(x) \}, \\
	R & = R' \setminus S', \qquad
	S =  S' \setminus R', \qquad
	T =  R' \cap S'.
\end{align*}
We have
\begin{align}
|R| + |S| + 2|T| + |W| = |N^c(x)| \ge \delta.	\label{eqn:2-factor:r+s+2t+w}
\end{align}
Note that $R' \cup S' \subseteq V(H)$.
If $y \in R'$, then $y_- \in N^c_{{\rm (b)}}(x) \cup N^c_{{\rm (d)}}(x) \cup N^c_{{\rm (e)}}(x) \cup  N^c_{{\rm (f)}}(x) \subseteq N^c(x)$.
For $y_- \in N^c_{{\rm (b)}}(x) $, we have $c(y_- y_{--}) \ne c(y_- y ) = c(x y_-)$ since $C_{y_-}$ is a properly coloured cycle.
Also, for $y_- \in N^c_{{\rm (d)}}(x) \cup N^c_{{\rm (e)}}(x) \cup  N^c_{{\rm (f)}}(x) $, we have $c(x y_-) \ne c(y_- y_{--})$ since $y_- \notin N^c_{\rm (c)}(x)$.
Thus, together with the definitions of $S'$ and~$N^{c}_{\rm (c)}(x)$, we deduce that
\begin{itemize}
	\item[(i)]	if $y \in R'$, then $y_- \in N^c(x)$ and $c(y_- y_{--}) \ne c(x y_-)$, and 
	\item[(ii)] if $y \in S'$, then $y_+ \in N^c(x)$ and $c(y_+ y_{++} ) \ne c(y_+y) = c(x y_+)$.
\end{itemize}

Let $F$ be the edge-coloured subgraph of $G$ induced by $W \cup R \cup S \cup T$.
For $y \in V(F)$, define the vertex colour $c(y)$ to be 
\begin{align*}
	c(y) = \begin{cases}
	c(xy) & \textrm{if } y \in W, \\
	c(y y_+)  & \textrm{if } y \in R, \\
	c(y y_-)  & \textrm{if } y \in S, \\
	c_0	& \textrm{if } y \in T,
	\end{cases}
\end{align*}
where $c_0$ is a new colour that does not appear in $G$.
The following claim concerns the colours of the edges of $F$.

\begin{clm} \label{clm:E(F)}
If $yz \in E(F)$, then $c(yz) = c(y)$ or $c(yz) = c(z)$.
In particular, $T$ is an independent set in $G$.
\end{clm}

\noindent
\textit{Proof of Claim~$\ref{clm:E(F)}$.}
Suppose the claim is false, so there exists an edge $yz \in E(F)$ with $c(yz) \ne c(y)$ and $c(yz) \ne c(z)$.
We will show that there is a properly coloured subgraph $H'$ in $G$ consisting of vertex-disjoint cycles with $V(H') = V(H) \cup \{x,y,z\}$, contradicting the maximality of $|H|$.

If $y,z \in W \subseteq N^c(x) $, then $c(xy) \ne c(xz)$.
Since $c(xy) = c(y) \ne c(yz) \ne c(z) = c(xz)$, then $xyzx$ is a properly edge-coloured triangle in~$G$.
Hence $H' = H + xyzx$ contradicts the maximality of $|H|$.

If $y \in W$ and $z \in R$, then $c(xy) = c(y) \ne c(yz) \ne c(z) = c(zz_+)$.
By~(i) we have $z_- \in N^c(x)$ and $c(x z_-) \ne c(z_- z_{--})$.
Moreover, we have $c(x z_-) \ne c(xy)$ since $y, z_- \in N^c(x)$.
Therefore, $C' = x y z C_z^+ z_- x$ is a properly coloured cycle.
We obtain a contradiction by setting $H' = H - C_z + C'$.
Similarly, if $y \in W$ and $z \in S$, then we obtain a contradiction by considering $C'' = x y z C_z^- z_+ x$ instead of~$C'$.
If $y \in W$ and $z \in T$, then $c(xy) = c(y) \ne c(yz)$.
Since $C_z$ is a properly coloured cycle, we have $c(yz) \ne c(z z_+)$ or $c(yz) \ne c(z z_-)$.
Recall that $T = R' \cap S'$, so $c(x z_-) \ne c(z_- z_{--})$ and $c(x z_+) \ne c(z_+ z_{++})$ by (i) and~(ii).
Also, $y,z_-, z_+ \in N^c(x)$, so $c(xz_-) \ne c(xy) \ne c(xz_+)$.
Hence, $H - C_z + C'$ or $H - C_z + C''$ would imply a contradiction.

Therefore, we may assume that $y,z \in R' \cup S' \subseteq V(H)$.
In order to prove the claim, it is enough to show that there exist at most two vertex-disjoint properly coloured cycles $C'$, $C''$ spanning $\{x\} \cup V(C_y) \cup V(C_z)$ (which would then imply that $H' = (H - C_y - C_z) + C' + C''$ is a union of properly coloured cycles with $|H'| = |H| +1$, a contradiction).

Suppose that $C_y$ and $C_z$ are distinct.
If $y, z \in R$, then by (i) we have 
\begin{align}
y_-, z_- & \in N^c(x), & c(y_- y_{--} ) & \ne c(xy_-), & c(x z_-) & \ne c(z_- z_{--}). \label{eqn:E(F)}
\end{align}
Since $y_-, z_- \in N^c(x)$, we have $c( x y_- ) \ne c( x z_- )$.
Note that $c(y) \ne c(yz) \ne c(z) $, so
\begin{align}
c(y y_+)  \ne c(yz) \ne   c(z z_+) . \label{eqn:E(F)2}
\end{align}
Hence, $C' = x y_- C^{-}_y y z C_z^+ z_- x$ (see Figure~\ref{fig:clm:E(F)}(A)) is a properly coloured cycle with vertex set $V(C_y + C_z) \cup \{ x \}$. 
Notice that $C'$ is a properly coloured cycle if both \eqref{eqn:E(F)} and \eqref{eqn:E(F)2} hold.
If $y,z \in R'$, then \eqref{eqn:E(F)} holds by~(i).
Therefore, we may assume without loss of generality that $y \in S$ or $y \in T$ with $c(y y_+) =  c(yz)$.
Since $T \subseteq S'$ and the cycle $C_y$ is properly coloured, we have $y \in S'$ and $c(y y_-) \ne c(y y_+) = c(yz)$.
By~(ii) and reversing the orientation of~$C_y$, we have $y_- \in N^c(x)$, $c(y_- y_{--} ) \ne c(xy_-)$ and $c(y y_+) \ne  c(yz)$.
Similarly, by reversing the orientation of $C_z$ if necessary, we may assume that both \eqref{eqn:E(F)} and \eqref{eqn:E(F)2} hold and so we derive a contradiction.
\begin{figure}[tbp]
\centering
\subfloat[]{
\includegraphics[scale=0.6]{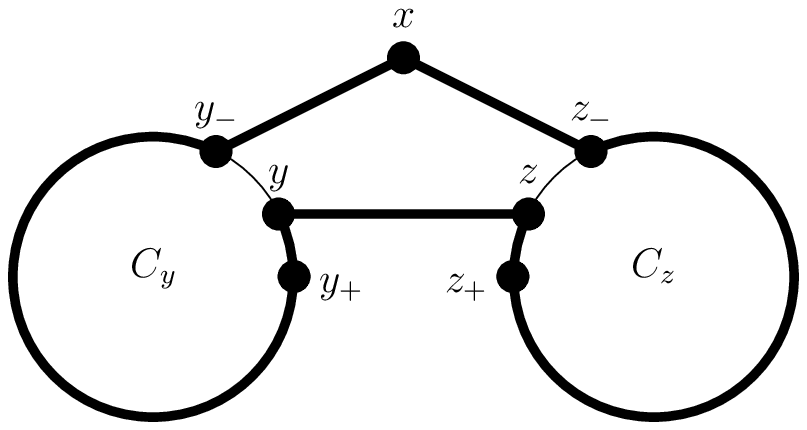}}
\subfloat[]{
\includegraphics[scale=0.6]{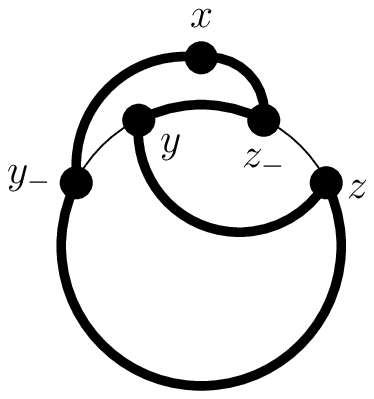}}
\subfloat[]{
\includegraphics[scale=0.6]{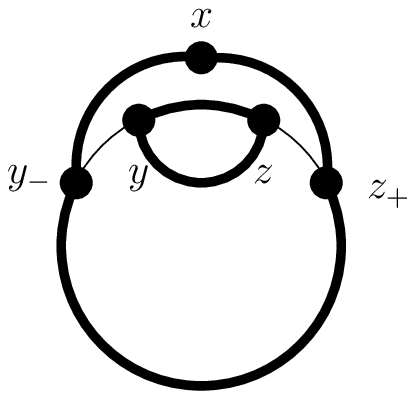}}
\caption{Properly coloured cycles used in the proof of Claim~\ref{clm:E(F)}.}
\label{fig:clm:E(F)}
\end{figure}

Finally, suppose that $C =C_y = C_z$.
If $y,z \in R$, then $c(y y_+) = c(y)  \ne c(yz) \ne c(z) = c(z z_+)$.
In particular, $y_+ \ne z$ and $z_+ \ne y$.
By~(i), $x y_- C^- z y C^+ z_- x$ (see Figure~\ref{fig:clm:E(F)}(B)) is a properly coloured cycle with vertex set $V(C)\cup \{ x \}$, so we are done.
If $y \in R$ and $z \in S$, then $y_- \ne z$, otherwise $x y C^+ z$ is a properly coloured cycle by (i) and~(ii).
However, both $x y_- C^- z_+ x$ and $y C^+ z y$ (see Figure~\ref{fig:clm:E(F)}(C)) are properly coloured cycles spanning $\{x\} \cup V(C)$.
If $y, z \in S$, then by reversing the orientation of the cycle~$C$ we conclude that $y,z \in R$ and so we are done.
If $y \in T$ or $z \in T$, then we apply the following `blindness' argument. 
Suppose that $y \in T$.
Since the cycle $C$ is properly coloured, $c(yz) \ne c(y y_+)$ or $c(yz) \ne c(y y_-)$.
We treat $y$ to be in $R$ if $c(yz) \ne c(y y_+)$, and $y$ to be in $S$ otherwise. 
We apply a similar treatment when $z \in T$.
Hence, we are back in the case when $y, z\in R \cup S$.
This completes the proof of Claim~\ref{clm:E(F)}. 
$\hfill{\blacksquare}$
\medskip

\begin{clm} \label{clm:y'}
For each vertex $y \in V(F)$ there exists a vertex $u = u(y) \notin V(F)$ such that either $yu \notin E(G)$ or $c(yu) = c(y)$.
\end{clm}

\noindent
\textit{Proof of Claim~$\ref{clm:y'}$.}
If $y \in W$, then $c(y) = c(xy)$ and so we are done by setting $u(y) = x$.

If $y \in R'$ with $y_- \in N^c_{{\rm (d)}}(x)$, then $xy$ is not an edge.
Set $u(y) = x$ for $y \in R$ with $y_{-} \in N^c_{(d)}(x)$.

If $y \in R'$ with $y_- \in N^c_{{\rm (f)}}(x)$, then $c(xy_-) = c(xz)$ for $z \in V(C_{y_-})$ by the remark after the definition of $N^c(x)$.
Hence, $V(C_y) \cap N^c(x) = \{y_-\}$.
This implies that $V(C_y) \cap V(F) = \{y\}$ and so $ y_+ \notin V(F)$.
Set $u(y) = y_+$ for $y \in R$ with $y_{-} \in N^c_{(f)}(x)$.

Suppose that $y \in R'$ with $y_-\notin N^c_{{\rm (d)}}(x) \cup N^c_{{\rm (f)}}(x)$ (i.e. $y_-
 \in N^c_{{\rm (b)}}(x) \cup N^c_{{\rm (e)}}(x)$).
We may assume that $xy \in E(G)$ or else we set $u(y) = x$.
If $y_{-} \in N^c_{{\rm (e)}}(x)$, then $c(x y ) \ne c(x y_{-})$ by definition.
If $y_- \in N^c_{{\rm (b)}}(x)$, then $c(xy_{-}) = c(yy_{-})$.
Recall that every monochromatic subgraph of $G$ is a disjoint union of stars. 
Hence, $c(x y) \ne c(x y_{-})$.
In summary, we have $c(x y)  \ne c(x y_{-})$  for $y \in R'$ with $y_- \notin N^c_{{\rm (d)}}(x) \cup N^c_{{\rm (f)}}(x)$.
If $c(x y) \ne c(y y_+)$, then (i) implies that $x y C_{y}^+ y_{-} x$ is a properly coloured cycle and so we can enlarge $H$, a contradiction.
Hence, 
\begin{align}
c(x y) = c(y y_+) \text{ for all $y \in R'\cap N_G(x)$ with $y_- \notin N^c_{{\rm (d)}}(x) \cup N^c_{{\rm (f)}}(x)$.} \label{eqn:y'1}
\end{align}
Set $u(y) = x$ for all $y \in R \cap N_G(x)$ with $y_- \notin N^c_{{\rm (d)}}(x) \cup N^c_{{\rm (f)}}(x)$.

Suppose that $y \in S'$, so (ii) implies that $c(y_+y_{++}) \ne c(xy_+)$.
We may assume that $xy \in E(G)$ or else we set $u(y) = x$.
Recall that every monochromatic subgraph of $G$ is a disjoint union of stars. 
Hence, $c(x y) \ne c(x y_{+})$.
If $c(x y) \ne c(y y_-)$, then $x y_+ C_{y}^+ y_{-} x$ is a properly coloured cycle and so we can enlarge $H$, a contradiction.
Therefore
\begin{align}
c(xy) & = c(yy_-) \textrm{ for all $y \in S'\cap N_G(x)$.} \label{eqn:y'2}
\end{align}
Hence, we set $u(y) = x$ for all $y \in S\cap N_G(x)$.

Finally, suppose that $y \in T = R' \cap S'$.
We may further assume that $y_- \notin N^c_{{\rm (d)}}(x) \cup N^c_{{\rm (f)}}(x)$ and $xy \in E(G)$.
Since $y \in R'\cap N_G(x)$ and $y_- \notin N^c_{{\rm (d)}}(x) \cup N^c_{{\rm (f)}}(x)$, we have $c(xy) = c(y y_+)$ by~\eqref{eqn:y'1}.
On the other hand, since $y \in S'\cap N_G(x)$, we have $c(x y) = c(y y_-)$ by~\eqref{eqn:y'2}. 
Hence, $c(y y_+) = c(y y_-)$ contradicting the fact that the cycle $C_y$ is properly coloured.
Therefore, $xy \notin E(G)$ for all $y \in T$, so we can set $u(y) = x$ for $y \in T$.
This completes the proof of Claim~\ref{clm:y'}.
$\hfill{\blacksquare}$

\medskip
If $V(F) = T$, then $|T| \ge \delta/2$ by \eqref{eqn:2-factor:r+s+2t+w}.
By Claim~\ref{clm:E(F)}, $T$ is an independent set in $G$.
Claim~\ref{clm:y'} implies that for each $y \in T$ there exists a vertex $u(y) \notin V(F) \cup N(y)$ as the colour $c(y) = c_0$ does not appear in~$G$.
Therefore
\begin{align*}
	\delta \le d^c(y) \le |N_G(y)| \le n - |T| - |\{ u(y) \}| \le n-\delta/2-1,
\end{align*}
which gives $\delta <2n/3$ as required.
Thus, we may assume $V(F) \ne T$.
We are going to show that there exists a vertex $y_0 \in R \cup S \cup W$ such that 
\begin{align}
 |\{z \in V(F) \setminus\{y_0\}: y_0z \notin E(G) \textrm{ or } c(y_0 z) =c(y_0) \}| \ge (\delta -1)/2. \label{eqn:keyineq}
\end{align}
Define $F'$ to be the directed graph on $V(F)$ such that there is a directed edge from $y$ to $z$ if $yz \notin E(F)$ or $c(yz) = c(z)$.
Thus, in proving \eqref{eqn:keyineq}, it suffices to show that there exists a vertex $y_0 \in R \cup S \cup W$ with indegree at least $(\delta-1)/2$ in $F'$.
By Claim~\ref{clm:E(F)}, the base graph of $F'$ is complete.
Moreover, $\overrightarrow{yz}$ exists for all $y \in T$ and $z \in V(F)$ as the colour $c(y)=c_0$ does not appear in~$G$.
Hence, the number of directed edges into $R \cup S \cup W$ is at least $\binom{ |R \cup S \cup W| }{2} + |T| |R \cup S \cup W|$.
By an averaging argument, there exists $y_0 \in R \cup S \cup W$ with indegree
\begin{align*}
	d^-(y_0) & \ge (|R|+|S|+|W|-1)/2 + |T| 
	\overset{\eqref{eqn:2-factor:r+s+2t+w}}{\ge}  (\delta - 1)/2, \nonumber
\end{align*}
so \eqref{eqn:keyineq} holds.
Recall that the colour degree $d^c(y_0) \ge \delta^c(G) = \delta $, so $y_0$ meets at least $\delta$ edges of distinct colours. 
In particular, there are at least $\delta-1$ neighbours $z'$ of $y_0$ with $c(y_0 z') \ne c(y_0)$.
Hence, there are at most $n-\delta$ vertices $z$ in $V(G) \setminus\{y_0\}$ such that either $y_0z \notin E(G)$ or $c(y_0 z)= c(y_0)$.
By Claim~\ref{clm:y'}, there exists $u = u(y_0) \in V(G) \setminus V(F)$ such that $y_0 u \notin E(G)$ or $c(y_0 u) = c(y_0)$.
Together with~\eqref{eqn:keyineq}, we have
\begin{align*}
	n-\delta & \ge  |\{z \in V(F)\setminus\{y_0\} : y_0 z \notin E \textrm{ or } c(y_0 z) =c(y_0) \}| +| \{ u(y_0)\}| \\
& \ge  (\delta -1)/2 + 1.
\end{align*}
Thus, $\delta < 2n/3$ as required.
This completes the proof of Theorem~\ref{thm:2factor}.
\end{proof}

The properly coloured $2$-factor obtained by Theorem~\ref{thm:2factor} may contain $|G|/3$ cycles.
We would like to minimise the number of cycles in a properly coloured $2$-factor.
In the next lemma, we show that this can be achieved by assuming a slightly larger $\delta^c_1(G)$.
Recall that $\delta_1^c(G)$ is the minimum $\delta(G - H)$ over all monochromatic subgraphs~$H$ in~$G$.
So $\delta^c_1(G) \ge \delta^c(G) - 1$.

\begin{lma} \label{lma:2factor1}
For every integer $k \ge 1$, every edge-coloured graph $G$ with $|G| = n $ and $ \delta^c_1(G) \ge 2n/3 + k$ contains a properly coloured $2$-factor in which every cycle has length at least $k/2$.
In particular, $G$ can be covered by at most $\lfloor 2 n /k \rfloor$ vertex-disjoint properly coloured paths.
\end{lma}

The idea of the proof is rather simple.
Recall that a $1$-path-cycle is a vertex-disjoint union of at most one path $P$ and a number of cycles.
We consider a properly coloured $1$-path-cycle $H$ in $G$ such that every cycle has length at least $k/2$ and $|H|$ is maximal.
If we suppose that $V(H) = V(G)$, then we may assume that $H$ contains a path~$P$ as a component or else we are done.
Our aim is to show that there exists a properly coloured $2$-factor $H'$ in $G[V(P)]$ such that each cycle has length at least $k/2$.
Therefore $H - P + H'$ is the desired properly coloured $2$-factor.

\begin{proof}[Proof of Lemma~$\ref{lma:2factor1}$]
Suppose the contrary, and let $G$ be a counterexample with $|G| = n $ and $ \delta^c_1(G) \ge 2n/3 + k$. 
Let $H$ be a properly coloured $1$-path-cycle in $G$ such that every cycle has length at least $k/2$ and $|H|$ is maximal.
If $V(H) = V(G)$, then $H$ contains a path $P$ as a component or else we are done.
If $V(H) \ne V(G)$, then we may assume $H$ contains a path (maybe consisting of a single vertex).
Let $P = z_1 z_2 \dots z_{\ell}$ be the path in $H$.
We further assume that $H$ is chosen such that $\ell$ is maximal (subject to $|H|$ maximal).
Suppose that $\ell =1$ and so $P = z_1$.
Then $N(z_1) \subseteq V(H)$ follows by the maximality of $|H|$.
Let $y \in N(z_1)$, and orient $C_y$ into a directed cycle such that $c(y y_+) \ne c(z_1 y)$.
Then $(H - C_y - z_1 )+ z_1 y C_y^{+} y_-$ is a properly coloured $1$-path-cycle on vertex set $V(H)$ with path $z_1 y C_y^{+} y_-$, contradicting the maximality of~$\ell$.
Therefore, $\ell \ge 2$.

Let $N_1 = \{ x \in N_G(z_1): c(z_1 x) \ne c(z_1z_2) \}$ and $N_{\ell} = \{ x \in N_G(z_{\ell}): c(z_{\ell} x) \ne c(z_{\ell} z_{\ell-1 }) \}$.
So 
\begin{align}
|N_1| , |N_{\ell}| \ge \delta^c_1(G) \ge 2n/3 +k. \label{eqn:n1}
\end{align}
By the maximality of $|H|$, we have $N_1, N_{\ell} \subseteq V(H)$.
If $y \in N_1 \setminus V(P)$, then orient $C_y$ into a directed cycle such that $c(y y_+) \ne c(z_1 y)$.
Then $H - C_y - P + y_- C_y^{-} y P $ contradicts the maximality of~$\ell$.
Therefore, $N_1 \subseteq V(P)$ and similarly $N_{\ell} \subseteq V(P)$.
Moreover, 
\begin{align}
 \nonumber \ell = |V(P)| \ge |N_1| \overset{\eqref{eqn:n1}}{\ge}  2n/3+k.
\end{align}
If $z_{\ell} \in N_1 $ and $z_1 \in N_{\ell}$, then $C = z_1 \dots z_{\ell} z_1$ is a properly coloured cycle of length at least~$\ell \ge k$.
Hence, $H - P + C$ consists of vertex-disjoint properly coloured cycles each of length at least $k/2$.
Note that $H-P+C$ spans $V(H)$.
If $V(H) = V(G)$, then we are done.
If $V(H) \ne V(G)$, then together with a vertex in $V(G) \setminus V(H)$ we obtain a properly coloured $1$-path-cycle contradicting the maximality of $|H|$.
Therefore, we have $z_1 \notin N_{\ell}$ or $z_{\ell} \notin N_1$.
Let 
\begin{align}
N'_q = N_q \setminus ( \{z_1, z_2, \dots, z_{ \lceil k/2 \rceil} \} \cup \{z_{\ell}, z_{\ell-1}, \dots, z_{ \ell - \lceil k/2 \rceil+1} \}) \label{eqn:N'q}
\end{align}
for $q \in \{1, \ell \}$.
Since $z_1 \notin N_1$ and $z_\ell \notin N_{\ell}$, by~\eqref{eqn:n1}, we have 
\begin{align}
	|N'_1|, |N'_{\ell}| \ge 2n/3 +k - (2 \lceil k/2 \rceil -1)\ge 2n/3. \label{eqn:N'1}
\end{align}
For a given $x = z_i \in V(P)$, we write $x_+$ for $z_{i+1}$ if $i < \ell$, and $x_-$ for $z_{i-1}$ if $i>1$.
Set $c_-(x) = c(xx_{-})$ for $x \ne z_1$ and $c_+(x) = c(x x_{+})$ for $x \ne z_{\ell}$.
For distinct $x,y \in V(P)$, $x P y$ denotes the subpath of $P$ from $x$ to $y$.
If $x=y$, set $xPx = x$.
For $q \in \{1, \ell \}$, define
\begin{align*}
 	R_q  & = \{  x \in V(P) \setminus \{z_1 \}: x_{-} \in N'_q \textrm{ and } c(x_- z_q) = c_+(x_-)\}, \\
 	S_q' & = \{  x \in V(P) \setminus \{z_{\ell} \}: x_{+} \in N'_q \textrm{ and } c(x_{+} z_q) =  c_-(x_{+})\}, \\
 	S_q'' & = \{  x \in V(P) \setminus \{z_{\ell} \} : x_{+} \in N'_q \textrm{ and }  c_+(x_+) \ne c(x_{+} z_q) \ne  c_-(x_{+})\},\\
 	S_q & = S_q' \cup S_q''.
\end{align*}
Note that $S_q' \cap S_q'' = \emptyset$.
Since $P$ is properly coloured, for each $x \in N_1$, we have $c(x z_1) \ne c_+(x)$ or $c(x z_1) \ne  c_-(x)$.
Therefore, the three sets $\{x_- : x \in R_1\}$, $\{x_+ : x \in S'_1\}$, $\{x_+ : x \in S_1''\}$ are pairwise disjoint.
Moreover, the union of these three sets is precisely~$N_1'$.
A similar statement also holds for $N'_{\ell}$.
Hence, $|R_q| + |S_q| = |N_q'|$ for $q \in \{1, \ell\}$.
Note that if $z_{\ell} \in R_1 \cup S_1$, then $z_{\ell} \in R_1 \setminus S_1$.
Similarly if $z_1 \in R_{\ell} \cup S_{\ell}$, then $z_1 \in S_{\ell} \setminus R_{\ell}$.
For $q \in \{1, \ell \}$, we set
\begin{align*}
	W_q & = ( (R_q \cup S_q) \setminus \{z_1, z_{\ell} \} ) \cup \{ z_q\} , &
	T_q & = R_q \cap S_q, &
	U_q & = W_q \setminus T_q.
\end{align*}
Note that $z_q \notin R_q \cup S_q$ and $z_1, z_{\ell} \notin R_q \cap S_q$.
Hence, $z_q \in U_q$ and $|W_q| + |T_q| \ge |N_q'|$.
Moreover, for $q \in \{1, \ell\}$,
\begin{align}
	2|W_q|  & =	|W_q| + |T_q| + |U_q|  \ge  |N'_q| + |U_q| 
	\overset{\eqref{eqn:N'1}}{\ge}  |U_q| + 2n/3 \label{eqn:W_q}.
\end{align}

Now define an auxiliary directed bipartite graph $F$ as follows.
The vertex classes of $F$ are $W_1$ and $W_{\ell}$.
(Recall that even though $W_1$ and $W_{\ell}$ might not be disjoint, we treat $W_1$ and $W_{\ell}$ to be so in $F$.)
For $\{q,q'\} = \{1, \ell\}$ there is a directed edge in $F$ from $x \in W_q$ to $y \in W_{q'}$ if and only if $xy \in E(G)$ and one of the following statements holds:
\begin{itemize}
	\item[\rm (i)] $x \in R_q \setminus ( T_q  \cup \{ z_1, z_{\ell} \} ) $ and $c(xy) \ne c_+(x)$;
	\item[\rm (ii)] $x \in S_q \setminus ( T_q  \cup \{ z_1, z_{\ell} \} ) $ and $c(xy) \ne c_-(x)$;
	\item[\rm (iii)] $x \in T_q$;
	\item[\rm (iv)] $q = 1$, $x= z_{1}$ and $c(xy) \ne c(z_1z_2)$;
	\item[\rm (v)] $q = \ell$, $x= z_{\ell}$ and $c(xy) \ne c(z_{\ell} z_{\ell -1} )$.
\end{itemize}

We are going to show that there exists a vertex $u \in U_1 \cup U_{\ell}$ such that $u$ is in at least $k$ directed 2-cycles in $F$.
If $x \in R_1 \setminus T_1$, then $\overrightarrow{xy}$ exists for $y \in W_{\ell}$ if and only if $c(xy) \ne c_+(x)$.
Hence, the outdegree of $x$ in $F$ is 
\begin{align*}
	d^{+}_{F} (x) & \ge | \{ z \in N_G(x) : c(xz) \ne c_+(x) \} | + | W_{\ell}| - n \\
& \ge   \delta_1^c(G) + |W_{\ell}| - n \\
& \overset{\eqref{eqn:W_q}}{\ge} (2n/3 + k ) + (|U_\ell| + 2n/3)/2 - n 
= (|U_{\ell}| +2k) /2.
\end{align*}
A similar statement holds for $x \in (S_1 \setminus T_1)\cup \{z_1\} = U_1 \setminus R_1$.
In summary, for all $x \in U_1$,
\begin{align*}	
		d^{+}_{F} (x) & \ge (|U_{\ell}| +2k) /2 .
\end{align*}
Hence, there are at least $|U_1| (|U_{\ell}| +2k) /2$ directed edges from $U_1$ to~$W_{\ell}$, and similarly there are at least $|U_{\ell} | (|U_{1}| +2k) /2$ directed edges from $U_{\ell}$ to~$W_{1}$.
Note that if $\overrightarrow{xy}$ is a directed edge in $F$ with $x \in U_1$ and $y \in T_{\ell}$, then $xy \in E(G)$ and so $\{x,y\}$ forms a directed $2$-cycle in $F$ by~(iii).
Hence, if $\overrightarrow{xy}$ is a directed edge in $F$ not contained in a directed $2$-cycle with $x \in U_1$ and $y \in W_{\ell}$, then $y \in U_{\ell}$.
A similar statement holds for $\overrightarrow{xy}$ with $x \in U_{\ell}$ and $y \in W_{1}$.
Therefore, at most $|U_1||U_{\ell}|$ directed edges from $U_1 \cup U_{\ell}$ are not contained in a directed 2-cycle in~$F$.
The number of directed edges $\overrightarrow{xy}$ in $F$ such that $x \in U_1 \cup U_{\ell}$ and $\overrightarrow{xy}$ is contained in a directed $2$-cycle is at least
\begin{align}
	 & \frac{|U_1|( |U_{\ell}| +2k)}{2} + \frac{|U_{\ell} | (|U_{1}| +2k )}2 - |U_1| |U_{\ell}|   = k( |U_1| + |U_{\ell}| ). \nonumber
\end{align}
Hence, by an averaging argument, there exists a vertex $u \in U_1 \cup U_{\ell}$ that is in at least $k$ directed $2$-cycles in $F$.
Assume that $u \in U_1$.
There are at least $k$ vertices $w \in W_2 \setminus u $ such that $uw$ is a directed 2-cycle in~$F$.
Pick one such $w \in W_2$ such that the subpath $uPw$ in $G$ has length at least~$k/2$.
Similarly, if $u \in U_{\ell}$, then there is a vertex $w \in W_1$ such that $uw$ is a directed 2-cycle in~$F$
and the subpath $uPw$ in $G$ has length at least~$k/2$.
Without loss of generality, we may assume that $u \in W_1$ and $w \in W_{\ell}$.

If $u \in T_1 = R_1 \cap S_1$, then $u \in R_1 \setminus \{z_1, z_{\ell}\}$ with $c(uw) \ne c_+(u)$, or $u \in S_1 \setminus \{z_1, z_{\ell}\}$ with $c(uw) \ne c_-(u)$, as $P$ is a properly coloured path and $u \notin\{z_1, z_{\ell}\}$.
Together with (i), (ii) and (iv), we may assume that at least one of the following statements holds:
\begin{itemize}
	\item[(a$_1'$)] $u \in R_1 $ and $c(uw) \ne c_+(u)$;
	\item[(a$_2'$)] $u \in S_1 $ and $c(uw) \ne c_-(u)$; 
	\item[(a$_3'$)] $u = z_1$ and $c(z_1w) = c(uw) \ne c(z_1z_2)$.
\end{itemize}
If $u \in R_1$, then the definition of $R_1$ implies that $u_- \in N'_1$ and $c(z_1u_-) = c_+(u_-) \ne c_-(u_-)$.
Since $u_- \in N'_1$, $c(z_1z_2) \ne c(z_1u_-)$.
A similar statement also holds for $u \in S_1$, so at least one of the following statements holds:
\begin{itemize}
	\item[(a$_1$)] $u \in R_1 $, $c(uw) \ne c_+(u)$ and $c(z_1z_2) \ne c(z_1u_-) \ne c_-(u_-)$;
	\item[(a$_2$)] $u \in S_1 $, $c(uw) \ne c_-(u)$ and $c(z_1z_2) \ne c(z_1u_+) \ne c_+(u_+)$;
	\item[(a$_3$)] $u = z_1$ and $c(z_1w) \ne c(z_1z_2)$.
\end{itemize}
By symmetry, we can deduce that at least one of the following statements holds for~$w$:
\begin{itemize}
	\item[(b$_1$)] $w \in R_{\ell} $, $c(uw) \ne c_+(w)$ and $c(z_{\ell}z_{\ell-1}) \ne c(z_{\ell}w_-) \ne c_-(w_-)$;
	\item[(b$_2$)] $w \in S_{\ell} $, $c(uw) \ne c_-(w)$ and $c(z_{\ell}z_{\ell-1}) \ne c(z_{\ell}w_+) \ne c_+(w_+)$;
	\item[(b$_3$)] $w = z_{\ell}$ and $c(z_{\ell}u) \ne c(z_{\ell-1} z_{\ell})$.
\end{itemize}
We now claim that 
\begin{itemize}
\item[($\dagger$)] there exists a properly coloured $2$-factor $H'$ in $G[V(P)]$ such that each cycle has length at least $k/2$.
 \end{itemize}
If ($\dagger$) holds, then $H'' = H - P + H'$ is a disjoint union of properly coloured cycles each of length at least $k/2$ on vertex set $V(H)$.
If $V(H'') = V(G)$, then $H''$ is the desired properly coloured $2$-factor.
If $z \in V(G) \setminus V(H'')$, then $H'' \cup \{ z \}$ is a properly coloured 1-path-cycle contradicting the maximality of $|H| = |H''|$.
Therefore, in order to complete the proof, it suffices to prove~($\dagger$) for each combination of (a$_i$) and (b$_j$) for $1 \le i,j \le 3$.
We say that $u < w$ if $u = z_i$ and $w = z_j$ with $i < j$.
We would like to point out that Cases~2--6 are proved by similar arguments used in Case~1.
Since the explicit structures of $H'$ (see Figure~\ref{fig:r1sl}) are different, we include their proofs for completeness.
\begin{figure}[tbp]
\centering

\subfloat[$u \in R_1$ and $w = S_{\ell}$ with $u_- = w$ ]{
\includegraphics[scale=0.6]{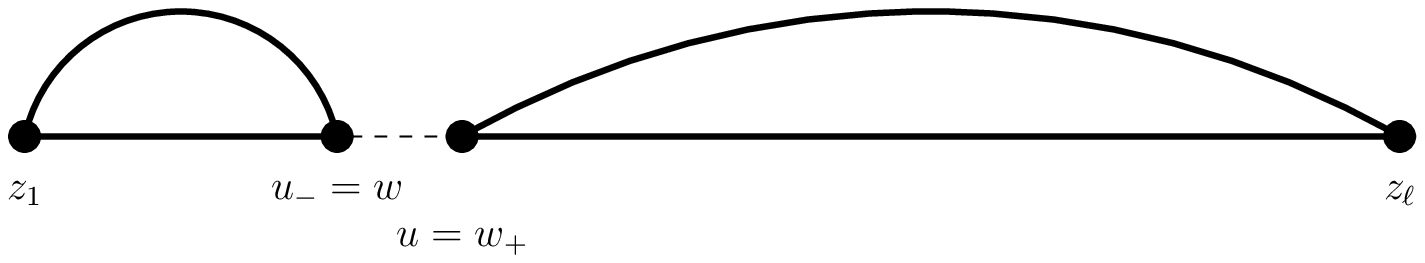}}

\subfloat[$u \in R_1$ and $w \in S_{\ell}$ with $u_- > w$]{
\includegraphics[scale=0.6]{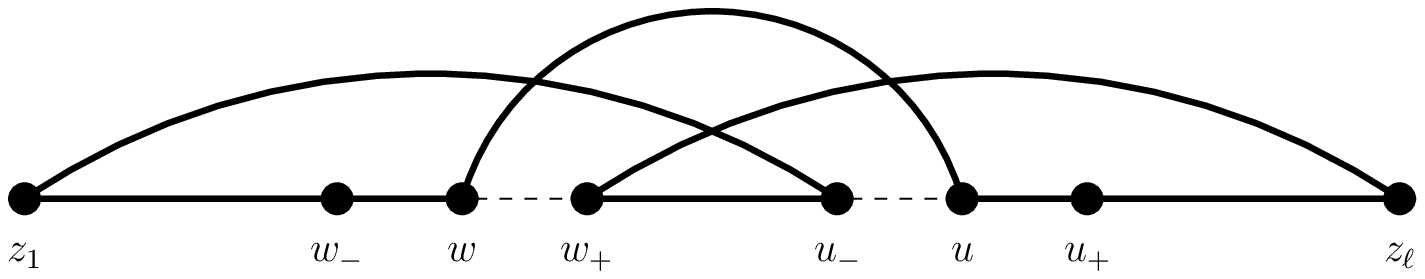}}

\subfloat[$u \in R_1$ and $w \in S_{\ell}$ with $u < w$]{
\includegraphics[scale=0.6]{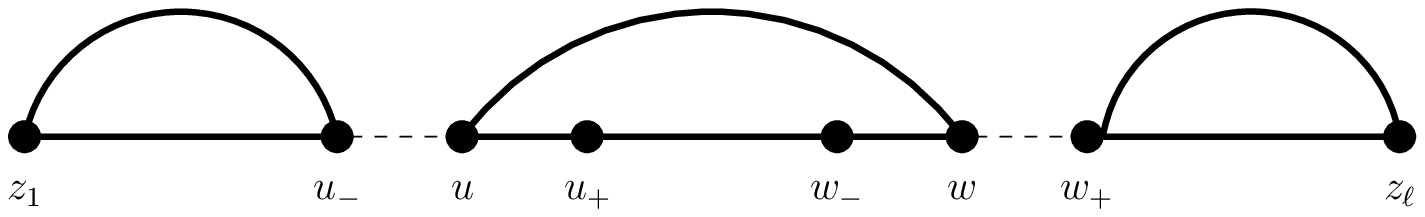}}

\subfloat[$u \in R_1$ and $w \in R_{\ell}$ with $u < w$]{
\includegraphics[scale=0.6]{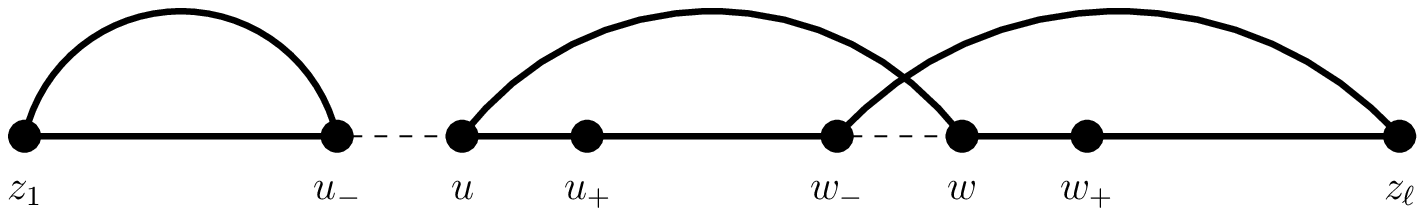}}

\subfloat[$u \in R_1$ and $w \in R_{\ell}$ with $u > w$]{
\includegraphics[scale=0.6]{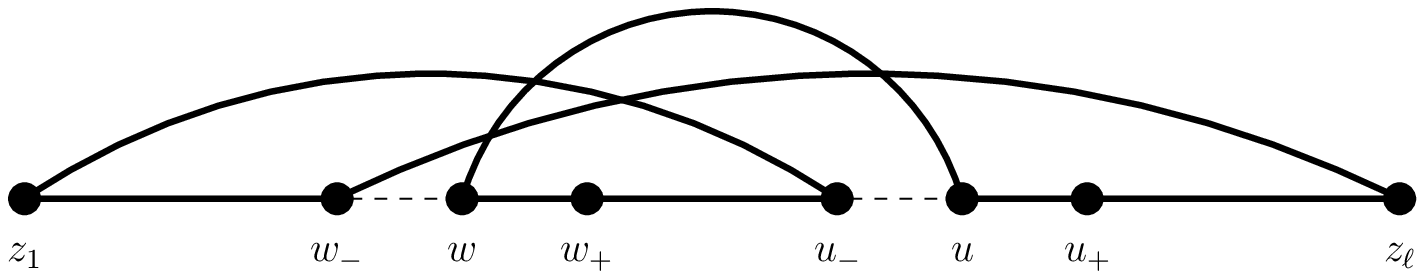}}

\subfloat[$u \in R_1$ and $w = z_{\ell}$]{
\includegraphics[scale=0.6]{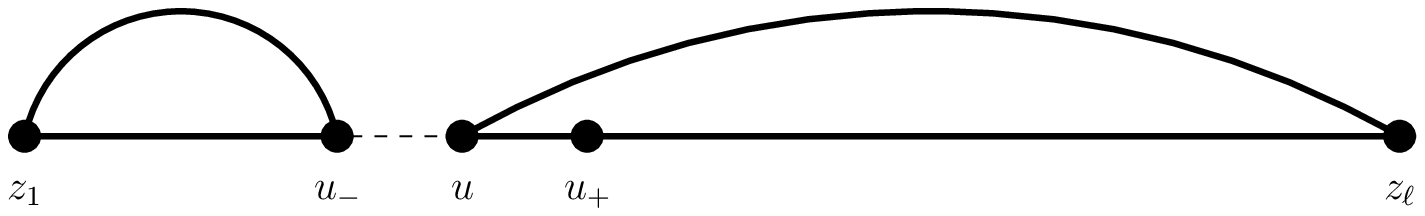}}

\subfloat[$u \in S_1$ and $w = R_{\ell}$ with $u < w_-$]{
\includegraphics[scale=0.6]{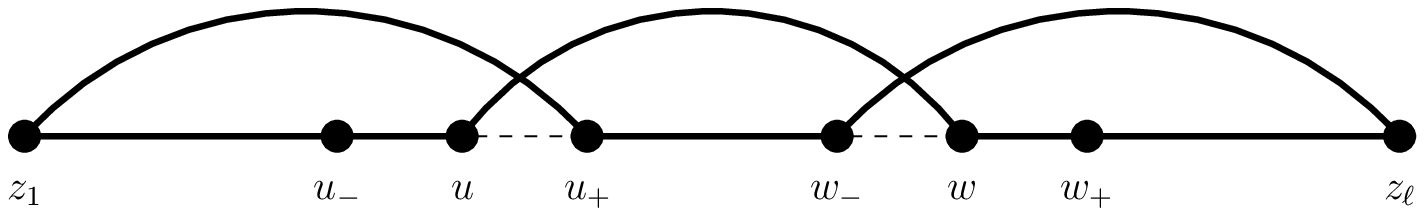}}

\subfloat[$u \in S_1$ and $w = R_{\ell}$ with $u =w_-$]{
\includegraphics[scale=0.6]{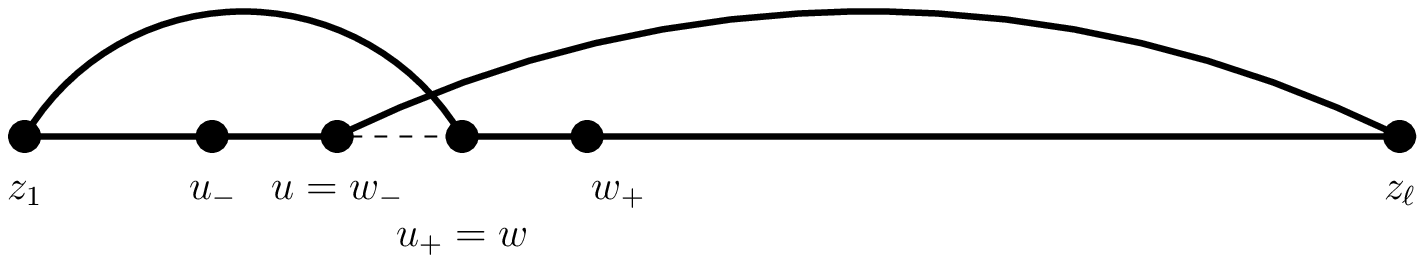}}

\subfloat[$u \in S_1$ and $w = R_{\ell}$ with $u > w$ ]{
\includegraphics[scale=0.6]{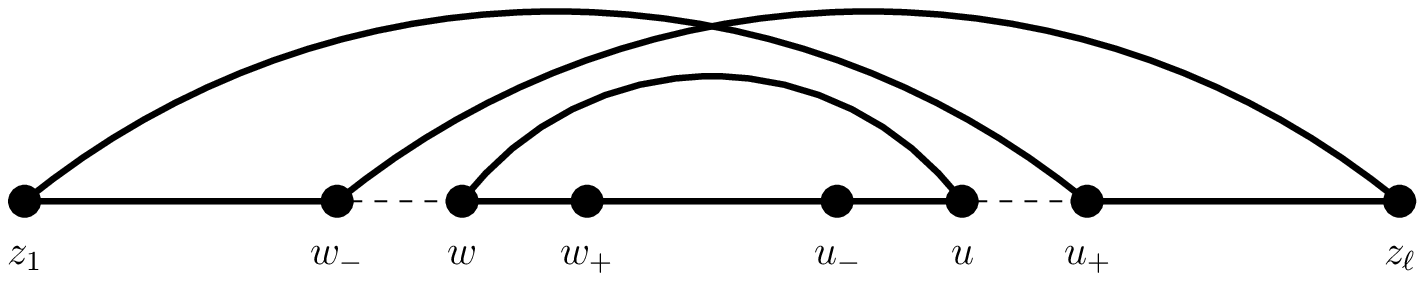}}

\subfloat[$u \in S_1$ and $w = z_{\ell}$]{
\includegraphics[scale=0.6]{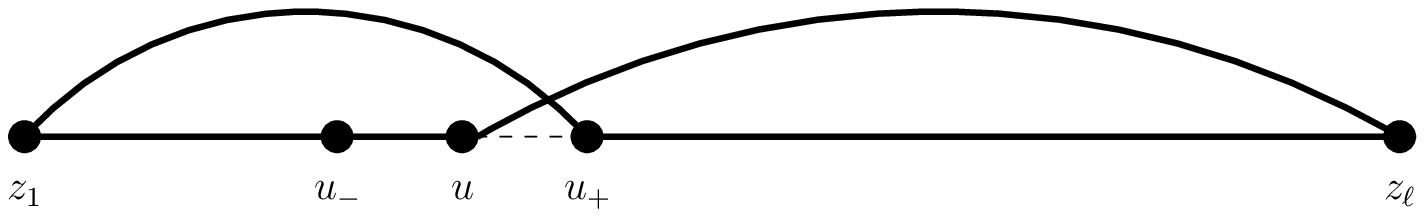}}

\caption{Structures of $H'$.}
\label{fig:r1sl}
\end{figure}
\noindent\textbf{Case 0: (a$_3$) and (b$_3$) hold.} 
Since $u = z_1$ and $w = z_{\ell}$, (a$_3$) implies that $z_{\ell} \in N_1$ and (b$_3$) implies that $z_1 \in N_{\ell}$.
Recall that $z_1 \notin N_{\ell}$ or $z_{\ell} \notin N_1$.
Hence, we have a contradiction.

\noindent\textbf{Case 1: (a$_1$) and (b$_2$) hold.}
We have the following three statements:
\begin{align}
c_+(u) & \ne c(uw) \ne  c_-(w), \label{eqn:cuw1}\\
c(z_1z_2) & \ne c( z_1u_-)  \ne  c_- (u_- ), \label{eqn:cuw2} \\
c(z_{\ell}z_{\ell-1}) & \ne c( z_{\ell} w_+) \ne  c_+ (w_+ ).  \label{eqn:cuw3}
\end{align}
If $u_- = w$, then we set $H' = z_1 P u_- z_1  + w_+ P z_{\ell} w_+$ (see Figure~\ref{fig:r1sl}(A)).
Note that $H'$ is properly coloured and a union of two cycles on vertex set $V(H)$.
By (a$_1$) and \eqref{eqn:N'q}, we have $u_- \in N'_1 \subseteq N_1 \setminus \{z_1, \dots, z_{\lceil k/2 \rceil} \}$.
Thus, $z_1 P u_- z_1$ is a cycle of length at least $k/2$.
By a similar argument, we deduce that $w_+Pz_{\ell} w_+$ is also a cycle of length at least $k/2$ as $w_+ \in N'_{\ell}$ by~(b$_2$).
Hence ($\dagger$) holds.

Next suppose that $u_- > w$.
If $u_- = w_+$ and $c(z_{\ell} w_+ ) = c(z_1 u_-)$, then
\begin{align}
c(z_{\ell} w_+ ) = c(z_1 u_-)  = c_+(u_-) = c_+(w_+)  \ne c_-(w_+), \nonumber
\end{align}
where the second equality is due to the fact that $u \in R_1$.
Recall~\eqref{eqn:cuw3} that $c( z_{\ell} w_+) \ne  c_+ (w_+ )$.
This contradicts the fact that $w \in S_{\ell}$.
Therefore, if $u_- = w_+$, then $c(z_1 u_-) \ne c(z_{\ell} w_+ )$.
Together with~\eqref{eqn:cuw1}--\eqref{eqn:cuw3}, we conclude that $H' = z_1 u_- P w_+ z_{\ell} P u w P z_1$ (see Figure~\ref{fig:r1sl}(B)) is a properly coloured cycle with vertex set $V(P)$, so ($\dagger$) holds.

If $u < w $, then $z_1 P u_- z_1$, $uPwu$ and $w_+Pz_{\ell} w_+$ (see Figure~\ref{fig:r1sl}(C)) are properly coloured cycles by~\eqref{eqn:cuw1}--\eqref{eqn:cuw3}.
We will show that each is a cycle of length at least $k/2$ (which then implies ($\dagger$)).
By (a$_1$) and \eqref{eqn:N'q}, we have $u_- \in N'_1 \subseteq N_1 \setminus \{z_1, \dots, z_{\lceil k/2 \rceil} \}$.
Thus, $z_1 P u_- z_1$ is a cycle of length at least $k/2$.  
Similarly, $w_+Pz_{\ell} w_+$ is also a cycle of length at least $k/2$ as $w_+ \in N'_{\ell}$ by~(b$_2$) and \eqref{eqn:N'q}.
Since $u$ and $w$ are chosen such that the subpath $uPw$ has length at least $k/2$, $uPwu$ has length at least $k/2$.
Therefore, ($\dagger$) holds for (a$_1$) and (b$_2$).

\noindent\textbf{Case 2: (a$_1$) and (b$_1$) hold.}
We have the following three statements:
\begin{align}
c_+(u) & \ne c(uw) \ne  c_+(w), \label{eqn:cuw111}\\
c(z_1z_2) & \ne c( z_1u_-)  \ne  c_- (u_- ), \label{eqn:cuw112} \\
c(z_{\ell}z_{\ell-1}) & \ne c( z_{\ell} w_-) \ne  c_- (w_- ).  \label{eqn:cuw113}
\end{align}
If $u < w$, then we set $H' = z_1 P u_- z_1  + u P w_- z_{\ell} P w u$ (see Figure~\ref{fig:r1sl}(D)).
Note that $H'$ is properly coloured.
By (a$_1$) and \eqref{eqn:N'q}, we deduce that $z_1 P u_- z_1$ is a cycle of length at least $k/2$.
Recall our choices of $u$ and $w$ that $uPw$ has length at least $k/2$.
This implies that $u P w_- z_{\ell} P w u$ is also a cycle of length at least $k/2$.
Hence ($\dagger$) holds.

Next suppose that $u > w$.
By~\eqref{eqn:cuw111}, we deduce that $w_+ \ne u$ and so $w < u_-$.
Hence, $w P u_-$ is a path with length at least one.
Together with~\eqref{eqn:cuw111}--\eqref{eqn:cuw113}, we conclude that $H' = z_1 P w_- z_{\ell} P u w P u_- z_1$ (see Figure~\ref{fig:r1sl}(E)) is a properly coloured cycle with vertex set $V(P)$, so ($\dagger$) holds. 
Therefore, ($\dagger$) holds for (a$_1$) and (b$_1$).

\noindent\textbf{Case 3: (a$_1$) and (b$_3$) hold.}
Since $w = z_{\ell}$, we have
\begin{align}
c_+(u) \ne c(u z_{\ell}) \ne  c_-(z_{\ell-1} z_{\ell})
\text{ and } c(z_1 z_2) \ne c(z_1 u_-) \ne c_-(u_-). \label{case2}
\end{align}
Note that both $z_1 P u_- z_1$ and $u P z_{\ell} u$ (see Figure~\ref{fig:r1sl}(F)) are properly coloured cycles.
By~\eqref{eqn:N'q}, both cycles have length at least $k/2$.
Therefore, ($\dagger$) holds for (a$_1$) and (b$_3$) by setting $H' = z_1 P u_{-} z_1 + u P z_{\ell} u$.

\noindent\textbf{Case 4: (a$_2$) and (b$_1$) hold.}
We have the following three statements:
\begin{align}
c_-(u) & \ne c(uw) \ne  c_+(w), \label{eqn:cuw211}\\
c(z_1z_2) & \ne c( z_1u_+)  \ne  c_+ (u_+), \label{eqn:cuw212} \\
c(z_{\ell}z_{\ell-1}) & \ne c( z_{\ell} w_-) \ne  c_- (w_- ).  \label{eqn:cuw213}
\end{align}
First suppose that $u < w_-$.
If $u_+ = w_-$ and $c(z_{\ell} w_- ) = c(z_1 u_+) $, then
\begin{align}
c(z_1 u_+) = c(z_{\ell} w_- )= c_+(w_-) = c_+(u_+) \ne c_-(u_+), \nonumber
\end{align}
where the second inequality is due to the fact that $w \in R_{\ell}$.
This contradicts the fact that $u \in S_{1}$ as $c( z_1u_+)  \ne  c_+ (u_+)$ by~\eqref{eqn:cuw212}.
Therefore, if $u_+ = w_-$, then $c(z_1 u_-) \ne c(z_{\ell} w_+ )$.
Together with~\eqref{eqn:cuw211}--\eqref{eqn:cuw213}, we conclude that $H' = z_1 P u w P z_{\ell} w_- P u_+ z_1$ (see Figure~\ref{fig:r1sl}(G)) is a properly coloured cycle with vertex set $V(P)$, so ($\dagger$) holds.

If $u = w_-$, then $H' = z_1 P w_- z_{\ell} P u_+ z_1$ (see Figure~\ref{fig:r1sl}(H)) is a properly coloured cycle with vertex set $V(P)$ by \eqref{eqn:cuw212} and~\eqref{eqn:cuw213}.
Thus ($\dagger$) holds.

Suppose that $u > w$.
By~\eqref{eqn:cuw211}, we deduce that $w_+ \ne u$ and so $w < u_-$.
Hence, $w P u w$ is a properly coloured cycle.
Moroever, it has length at least $k/2$ by our choices of $u$ and $w$.
By \eqref{eqn:cuw212} and \eqref{eqn:cuw213}, we conclude that $z_1 P w_- z_{\ell} P u_+ z_1$ is a properly coloured cycle.
It has length at least $k/2$ by~\eqref{eqn:N'q}.
Hence ($\dagger$) holds by setting $H' = w P u w+ z_1 P w_- z_{\ell} P u_+ z_1$ (see Figure~\ref{fig:r1sl}(I)).
Therefore, ($\dagger$) holds for (a$_2$) and (b$_1$).

\noindent\textbf{Case 5: (a$_2$) and (b$_3$) hold.}
Since $w = z_{\ell}$, we have
\begin{align}
c_-(u) \ne c(u z_{\ell}) \ne  c_-(z_{\ell-1} z_{\ell})
\text{ and } c(z_1 z_2) \ne c(z_1 u_+) \ne c_+(u_+). \nonumber
\end{align}
Note that $z_1 P u z_{\ell} P u_+ z_1$ (see Figure~\ref{fig:r1sl}(J)) is a properly coloured cycle with vertex set $V(P)$ and so ($\dagger$) holds.

\noindent\textbf{Case 6: (a$_i$) and (b$_j$) hold for $(i,j) \in \{(3,2), (2,2), (3,1)\}$.}
Let $P' = z'_1 z'_2 \dots z'_{\ell}$ be the properly coloured path obtained by reversing~$P$. 
Hence, $z'_i = z_{\ell - i + 1}$ for all $1 \le i \le \ell$.
Given $x' = z_i' \in V(P')$ with $i < \ell$, write $x_{+'}$ to mean $z'_{i+1}$ and  $c_{+'}(x) = c(x x_{+'})$.
Similarly, given $x' = z_i' \in V(P')$ with $i >1 $, we write $x_{-'}= z'_{i-1}$ and $c_{-'}(x) = c(x x_{-'})$.
(Hence, these primed notations are defined relative to~$P'$.)
Recall that $P'$ is the reversed $P$, so $c_{+'}(x) = c_-(x)$ and $c_{-'}(x) = c_+(x)$.
Further, set $u' = w$ and $w' = u$.

Now suppose that  (a$_3$) and (b$_2$) hold.
Under the primed notation, we have
\begin{align*}
c_{+'}(u') \ne c(u' z'_{\ell}) \ne  c_-(z'_{\ell-1} z'_{\ell})
\text{ and } c(z'_1 z'_2) \ne c(z'_1 u'_{-'}) \ne c_{-'}(u'_{-'}).
\end{align*}
More importantly, if we ignore the primes, then we obtain \eqref{case2}.
Therefore, in hindsight, we are in the case when (a$_1$) and (b$_3$) hold with respect to~$P'$, i.e. we are in Case~3.
Hence, ($\dagger$) holds by setting $H' = z'_1 P' u'_{-'} z'_1 + u' P' z'_{\ell} u'$.

Similarly, we deduce that the case when (a$_2$) and (b$_2$) hold with respect to $P$ corresponds to the case when (a$_1$) and (b$_1$) hold with respect to $P'$.
Also, (a$_3$) and (b$_1$) (with respect to $P$) corresponds to (a$_2$) and (b$_3$) (with respect to~$P'$).
Therefore, we are in Case~1 and Case~5, respectively.
The proof of the lemma is completed.
\end{proof}

\section{absorbing cycle} \label{sec:abscycle}

The aim of this section is to prove the following lemma, that is, to show that there exists a small absorbing cycle.

\begin{lma}[Absorbing cycle lemma] \label{lma:abscycle}
Let $0 < \eps < 2^{-9} 3^{-1}$.
Then there exists an integer $n_0$ such that whenever $n \ge n_0$ the following holds.
Suppose that $G$ is an edge-coloured graph of order $n$ with $\delta^c_1(G) \ge (2/3 + \eps) n$.
Then there exists a properly coloured cycle $C$ of length at most $2 \eps n /3$ such that for all $k \le   (8 \eps^2 n )/243$ and for any collections $P_1,\dots,P_k$ of vertex-disjoint properly coloured paths in $G \setminus V(C)$, there exists a properly coloured cycle with vertex set $V(C) \cup \bigcup_{1 \le i \le k} V(P_i)$.
\end{lma}

Given a vertex $x$, we say that a path $P$ is an \emph{absorbing path for $x$} if the following conditions hold:
\begin{enumerate}
	\item[\rm (i)] $P = z_1z_2z_3z_4$ is a properly coloured path of length~$3$; 
	\item[\rm (ii)] $ x \notin V(P)$;
	\item[\rm (iii)] $z_1 z_2 x z_3 z_4$ is a properly coloured path.
\end{enumerate}
Next we define an absorbing path for two disjoint edges.
Given two vertex-disjoint edges $x_1x_2$, $y_1y_2$, we say that a path $P$ is an \emph{absorbing path for $(x_1, x_2; y_1, y_2)$} if the following conditions hold:
\begin{enumerate}
	\item[\rm (i)] $P = z_1z_2z_3z_4$ is a properly coloured path of length~$3$;
	\item[\rm (ii)] $V(P) \cap  \{ x_1, x_2, y_1, y_2\} = \emptyset$; 
	\item[\rm (iii)] both $z_1 z_2 x_1 x_2 $ and $y_1 y_2 z_3 z_4$ are properly coloured paths of length~3.
\end{enumerate}
Note that the ordering of $(x_1, x_2; y_1, y_2)$ is important.
Given a vertex $x$, let $\mathcal{L}(x)$ be the set of absorbing paths for $x$.
Similarly, given two vertex-disjoint edges $x_1x_2$, $y_1y_2$, let $\mathcal{L}(x_1, x_2; y_1,y_2)$ be the set of absorbing paths for $(x_1, x_2; y_1,y_2)$.
The following simple proposition follows immediately from the definition of an absorbing path for $(x_1, x_2; y_1,y_2)$.

\begin{prp} \label{prp:abspath}
Let $P' = x_1 x_2 \dots x_{\ell-1} x_{\ell}$ be a properly coloured path with $\ell \ge 4$. 
Let $P = z_1 z_2 z_3 z_4$ be an absorbing path for $(x_1, x_2; x_{\ell-1}, x_{\ell} )$ with $V(P) \cap V(P') = \emptyset$.
Then $z_1 z_2 x_1 x_2 \dots x_{\ell -1} x_{\ell} z_3z_4$ is a properly coloured path.
\end{prp}

Lemma~\ref{lma:abscycle} will be proved as follows.
Suppose that $\delta_1^c(G) \ge (2/3 + \eps)n $. 
In the next lemma, we show that every $\mathcal{L}(x)$ and every $\mathcal{L}(x_1, x_2; y_2, y_1)$ are large.
By a simple probabilistic argument, Lemma~\ref{lma:absorbing} shows that there exists a small family $\mathcal{F}'$ of vertex-disjoint properly coloured paths (each of length~$3$) such that $\mathcal{F}'$ contains a linear number of members of  $\mathcal{L}(x)$ for all $x \in V(G)$ and a similar statement holds for $\mathcal{L}(x_1, x_2; y_1, y_2)$.
Finally, we join all paths in $\mathcal{F}'$ into one short properly coloured cycle $C$ using Lemma~\ref{lma:ifar}.
Moreover, $C$ satisfies the desired property in Lemma~\ref{lma:abscycle}.

\begin{lma} \label{lma:abspath}
Let $0 < \eps < 1/2$ and let $n \ge  4/ \eps^{2}$ be an integer.
Suppose that $G$ is an edge-coloured graph on $n$ vertices with $\delta^c_1(G) \ge (1/2+\eps) n$.
Then $|\mathcal{L}(x)|\ge \eps^3 n^4 $ for every $x \in V(G)$ and $|\mathcal{L}(x_1, x_2; y_1, y_2)|\ge \eps^3 n^4 $ for any distinct vertices $x_1,x_2, y_1,y_2 \in V(G)$ with $x_1x_2$, $y_1y_2 \in E(G)$.
\end{lma}

\begin{proof}
Fix a vertex $x \in V(G)$.
Choose a vertex $z_2 \in N(x)$.
Next pick another vertex $z_3 \in N(z_2) \cap N(x)$ such that $c(x z_2) \ne c(x z_3)$.
Notice that the number of such $z_3$ is at least $\delta^c_1(G) + \delta(G) - n \ge 2 \eps n$.
Since $\delta_1^c(G) \ge ( 1/2 + \eps ) n$, $\Delta(H) \le ( 1/2- \eps) n $ for all monochromatic subgraphs~$H$ in~$G$.
Hence, the number of $z_1 \in N(z_2) \setminus \{x, z_3\}$ such that $c(z_2z_3) \ne c(z_1z_2) \ne c(z_2x)$ is at least $ \delta^c_1(G) - ( 1/2- \eps) n -2 \ge \eps n$.
Fix one such~$z_1$.
Similarly, there are at least $\eps n$ choices for $z_4 \in N(z_3) \setminus \{x, z_1, z_2\}$ such that $c(z_2z_3) \ne c(z_3z_4) \ne c(x z_3)$.
Notice that $z_1 z_2 z_3 z_4$ is an absorbing path for $x$.
Recall that every path is assumed to be directed.
Therefore, there are at least $\delta^c_1(G) \times 2 \eps n \times \eps n \times \eps n \ge \eps^3 n^4$ many absorbing paths for $x$.

Next, fix vertex-disjoint edges $x_1 x_2$ and $y_1 y_2$ in $G$.
Choose a vertex $z_2 \in N(x_1) \setminus \{ x_2,y_1,y_2\}$ such that $c(x_1 z_2 ) \ne c(x_1x_2)$.
Pick another vertex $z_3 \in ( N(z_2) \cap N(y_2) ) \setminus \{ x_1,x_2,y_1,y_2\}$ such that $c(y_2z_3) \ne c(y_1 y_2)$.
The number of such $z_3$ is at least $\delta(G)+ \delta^c_1(G) - n -4 \ge 2 \eps n -4$.
Recall that $\Delta(H) \le ( 1/2- \eps) n $ for all monochromatic subgraphs~$H$ in~$G$.
Hence, the number of $z_1 \in N(z_2) \setminus \{x_1 ,x_2,y_1,y_2, z_3\}$ such that $c(z_2z_3) \ne c(z_1z_2) \ne c(z_2x_1)$ is at least $ \delta^c_1(G) - ( 1/2- \eps) n - 4  \ge  \eps n$.
Fix one such~$z_1$.
Similarly, there are at least $ \delta^c_1(G) - ( 1/2- \eps) n -  5 \ge  \eps n$ choices for $z_4 \in N(z_3) \setminus \{x_1,x_2,y_1,y_2, z_1, z_2\}$ such that $c(z_2z_3) \ne c(z_3z_4) \ne c(y_2 z_3)$.
Note that $z_1 z_2 z_3 z_4$ is an absorbing path for $(x_1, x_2 ; y_1, y_2)$.
Therefore, there are at least 
\begin{align*}
 (\delta^c_1(G)-4) (2 \eps n-4) \eps^2 n^2 \ge ( \eps n^2 + 2n (\eps^2 n-6 \eps -1 )+ 16  )\eps^2 n^2 \ge\eps^3 n^4
\end{align*}
 many absorbing paths for $(x_1, x_2 ;y_1, y_2)$.
\end{proof}

Lemma~\ref{lma:absorbing} is proved by a simple probabilistic argument since each of $\mathcal{L}(x)$ and $\mathcal{L} (x_1,x_2; y_1, y_2)$ is large.
We will need the following Chernoff bound for the binomial distribution (see e.g.~\cite{MR1885388}).
Recall that the binomial random variable with parameters $(n,p)$ is the sum
of $n$ independent Bernoulli variables, each taking value $1$ with probability~$p$,
or $0$ with probability $1-p$.

\begin{prp}\label{prop:chernoff}
Suppose that $X$ has the binomial distribution and $0<a<3/2$. Then
$\mathbb{P}(|X - \mathbb{E}X| \ge a\mathbb{E}X) \le 2 e^{-a^2\mathbb{E}X /3}$.
\end{prp}

\begin{lma} \label{lma:absorbing}
Let $0 < \gamma < 1$.
Then there exists an integer $n_0$ such that whenever $n \ge n_0$ the following holds.
 Let $G$ be an edge-coloured graph on $n$ vertices.
Suppose that $|\mathcal{L}(x)|\ge \gamma n^4 $ for every $x \in V(G)$ and $|\mathcal{L} (x_1,x_2; y_1, y_2)| \ge \gamma n^4 $ for all distinct vertices $x_1,x_2, y_1,y_2 \in V(G)$ with $x_1x_2$, $y_1y_2 \in E(G)$.
Then there exists a family $\mathcal{F}'$ of vertex-disjoint properly coloured paths each of length~$3$, which satisfies the following properties:
\begin{align*}
|\mathcal{F}'| & \le 2^{-6} \gamma n, \\
|\mathcal{L}(x) \cap \mathcal{F}'| & \ge 2^{-9} \gamma^2 n,  \\
|\mathcal{L}(x_1, x_2;y_1, y_2) \cap \mathcal{F}'| & \ge 2^{-9} \gamma^2 n
\end{align*}
for all $x \in V(G)$ and for all distinct vertices $x_1,x_2, y_1,y_2 \in V(G)$ with $x_1x_2$, $y_1y_2 \in E(G)$.
\end{lma}

\begin{proof}
Choose $n_0 \in \mathbb{N}$ large so that 
\begin{align}
\exp (- \gamma n_0 / (3 \times 2^7)) + ( n_0 + n_0^4 ) \exp (- \gamma^2 n_0 / (3 \times 2^9))\le 1/6.
\label{eqn:n0}
\end{align}
Recall that each path is assumed to be directed.
So a path $z_1 z_2 z_3 z_4$ will be considered as a $4$-tuple $(z_1, z_2, z_3, z_4)$.
Choose a family $\mathcal{F}$ of $4$-tuples in $V(G)$ by selecting each of the $n!/ (n-4)!$ possible $4$-tuples independently at random with probability 
\begin{align*}
p =  2^{-7}\gamma \frac{ (n-4 )!}{(n-1)!}\ge 2^{-7} \gamma n^{-3}.
\end{align*}
Notice that 
\begin{align*}
\mathbb{E}|\mathcal{F}|  & = p \frac{n!}{(n-4)!} = 2^{-7} \gamma n,\\
\mathbb{E}|\mathcal{L}(x) \cap \mathcal{F}| & = p |\mathcal{L}(x)| \ge 2^{-7}\gamma^2 n, \\
\mathbb{E}|\mathcal{L}(x_1, x_2;y_1, y_2) \cap \mathcal{F}| & = p |\mathcal{L}(x_1, x_2;y_1, y_2)| \ge 2^{-7}\gamma^2 n
\end{align*}
for every $x \in V(G)$ and for all distinct $x_1,x_2, y_1,y_2 \in V(G)$ with $x_1x_2$, $y_1y_2 \in E(G)$.
Then by Proposition~\ref{prop:chernoff}, the union bound and~\eqref{eqn:n0} with probability at least~$5/6$, the family $\mathcal{F}$ satisfies the following properties:
\begin{align}
|\mathcal{F}| & \le 2 \mathbb{E}(|\mathcal{F}|)  = 2^{-6} \gamma n, \label{eqn:|F|}\\
|\mathcal{L}(x) \cap \mathcal{F}| & \ge 2^{-1} \mathbb{E}(|\mathcal{L}(x) \cap \mathcal{F}|)
\ge 2^{-8} \gamma^2 n,
\label{eqn:L1}\\
|\mathcal{L}(x_1, x_2;y_1, y_2) \cap \mathcal{F}| & \ge 2^{-1} \mathbb{E}( |\mathcal{L}(x_1, x_2;y_1, y_2) \cap \mathcal{F}|)
\ge 2^{-8} \gamma^2 n
\label{eqn:L2}
\end{align} 
for every $x \in V(G)$ and for all distinct $x_1,x_2, y_1,y_2 \in V(G)$ with $x_1x_2$, $y_1y_2 \in E(G)$.

We say that two $4$-tuples $(a_1, a_2,a_3,a_4)$ and $(b_1, b_2,b_3,b_4)$ are \emph{intersecting} if $a_i = b_j$ for some $1 \le i,j\le 4$.
Furthermore, we can bound the expected number of pairs of $4$-tuples in $\mathcal{F}$ that are intersecting from above by
\begin{align}
	\frac{n!}{(n-4)!} \times 4^2 \times \frac{(n-1)!}{(n-4)!} \times p^2 =  2^{-10} \gamma^2  n.
\nonumber
\end{align}
Thus, using Markov's inequality, we derive that with probability at least~$1/2$,
\begin{align}
	\textrm{$\mathcal{F}$ contains at most $ 2^{-9} \gamma^2 n$ intersecting pairs of $4$-tuples.} \label{eqn:F}
\end{align}
Hence, with positive probability the family $\mathcal{F}$ has all properties stated in \eqref{eqn:|F|}--\eqref{eqn:F}.
Remove one $4$-tuple in each intersecting pair in such a family~$\mathcal{F}$.
Further remove those $4$-tuples that are not absorbing paths.
We get a subfamily $\mathcal{F}'$ consisting of pairwise disjoint $4$-tuples, which satisfies
\begin{align}
|\mathcal{L}(x) \cap \mathcal{F}'| > &2^{-8} \gamma^2 n -2^{-9} \gamma^2 n= 2^{-9} \gamma^2 n \nonumber
\end{align}
for every $x \in V(G)$ and a similar statement holds for $|\mathcal{L}(x_1,x_2;y_1,y_2) \cap \mathcal{F}'|$.
Since each $4$-tuple in $\mathcal{F}'$ is an absorbing path, $\mathcal{F}'$ is a set of vertex-disjoint properly coloured paths of length~$3$.
\end{proof}

In order to prove Lemma~\ref{lma:abscycle}, it is sufficient to join the paths in $\mathcal{F}'$ given by Lemma~\ref{lma:absorbing} into a short properly coloured cycle $C$.
The next lemma shows that we can join any two disjoint edges into a properly coloured path of length~$5$.

\begin{lma} \label{lma:ifar}
Suppose that $G$ is an edge-coloured graph of order $n$ with $\delta^c_1(G) \ge 2n/3 +1$.
Let $x_1x_2$, $y_1y_2$ be two edges in $G$ with $x_2 \ne y_2$.
Then there exists an edge $z_1z_2$ with $z_1$, $z_2 \in V(G) \setminus \{ x_1, x_2, y_1, y_2\}$ such that both $x_1x_2z_1z_2$ and $z_1z_2y_1y_2$ are properly coloured paths. 
In particular, if $x_1$, $x_2$, $y_1$ and $y_2$ are distinct, then $x_1x_2z_1z_2y_1y_2$ is a properly coloured path.
\end{lma}

\begin{proof}
Let $X$ be the set of vertices $x \in N(x_2) \setminus  \{ x_1, x_2, y_1, y_2\}$ such that $c(x x_2) \ne c(x_1x_2)$.
Similarly, let $Y$ be the set of vertices $y \in N(y_1) \setminus  \{ x_1, x_2, y_1, y_2\}$ such that $c(y y_1) \ne c(y_1y_2)$.
So $|X|, |Y| \ge \delta^c_1(G) - 2 \ge 2n/3 -1$.
Define an auxiliary bipartite directed graph $H$ with vertex classes $X$ and $Y$ such that for $x \in X$ and $y \in Y$ 
\begin{itemize}
	\item[\rm (a)] $\overrightarrow{xy} \in E(H)$ if and only if $xy \in E(G)$ and $c(x y) \ne c(x x_2)$;
	\item[\rm (b)] $\overrightarrow{yx} \in E(H)$ if and only if $xy \in E(G)$ and $c(x y) \ne c(y y_1)$.
\end{itemize}
(Recall that in $H$ we treat $X$ and $Y$ to be disjoint.)
The outdegree of each $x \in X$ in $H$ is 
\begin{align*}
d^+_H(x) & = | Y \cap \{z \in N_G(x) :  c(xz) \ne c(x x_2)\}| \\
& \ge  |Y| + \delta^c_1(G)- n  \ge (|Y| + 1)/2
\end{align*}
and similarly for each $y \in Y$, $d^+_H(y) \ge (|X| + 1)/2$.
Hence the number of directed $2$-cycles in $H$ is at least
\begin{align*}
	\sum_{x \in X} d^+_H(x) + \sum_{y \in Y} d^+_H(y) - |X||Y|
		& \ge  \frac{|X|(|Y| + 1 )}2  + \frac{|Y|(|X| + 1)}2 - |X||Y| \\
	& = \frac{|X| +|Y| }2 \ge  2n/3 - 1 .
\end{align*}
Let $z_1z_2$ be a directed $2$-cycle in $H$ with $z_1 \in X$ and $z_2 \in Y$. 
This implies that $c(x_2z_1) \ne c(z_1z_2) \ne c(z_2 y_1)$.
Hence, $x_1 x_2 z_1 z_2 $ and $z_1 z_2 y_1 y_2$ are properly coloured paths.
Therefore, the lemma follows.
\end{proof}

We are ready to prove Lemma~\ref{lma:abscycle}.
Given a graph family $\mathcal{F}$, we write $V(\mathcal{F}) = \bigcup_{F \in \mathcal{F}} V (F)$.

\begin{proof}[Proof of Lemma~$\ref{lma:abscycle}$]
Let $\gamma = 2^6 \eps  /9$.
Choose $n_0 \in \mathbb{N}$ large so that $n_0 \ge 4/ \eps^2$ and Lemma~\ref{lma:abspath} holds.
By Lemma~\ref{lma:abspath}, $\mathcal{L}(x)\ge \gamma  n^4$ for all $x \in V(G)$ and 
$\mathcal{L}(x_1, x_2; y_1, y_2)\ge \gamma  n^4$ for all distinct vertices $x_1,x_2, y_1,y_2 \in V(G)$ with $x_1x_2$, $y_1y_2 \in E(G)$.
Let $\mathcal{F}'$ be the set of properly coloured paths obtained by Lemma~\ref{lma:absorbing}.
Therefore, $|\mathcal{F}'|  \le 2^{-6} \gamma n = \eps n /9$,
\begin{align}
 |\mathcal{L}(x) \cap \mathcal{F}'| & \ge  2^{-9} \gamma^2 n = \frac{8 \eps^2 n}{81} \text{ and } 
|\mathcal{L}(x_1, x_2; y_1, y_2) \cap \mathcal{F}'| & \ge  \frac{8 \eps^2 n}{81}& \label{eqn:L3}
\end{align}
for all $x \in V(G)$ and for all distinct vertices $x_1,x_2, y_1,y_2 \in V(G)$ with $x_1x_2$, $y_1y_2 \in E(G)$.

We now show that $C$ has the desired property.
Let $P_1, \dots, P_{|\mathcal{F}'|}$ be the properly coloured paths in $\mathcal{F}'$.
For each $1 \le j \le |\mathcal{F}'|$, we are going to find an edge $z_1^j z_2^j$ such that $\{z_1^j, z_2^j\} \cap V(\mathcal{F}') = \emptyset$, $P_j  z_1^j z_2^j P_{j+1}$ is a properly coloured path, where we take $P_{|\mathcal{F}'|+1} = P_1$, and $\{z_1^j, z_2^j\} \cap \{z_1^{j'}, z_2^{j'}\} = \emptyset$ for all $j \ne j'$.
Assume that we have already found edges $z_1^1 z_2^1,\dots,z_1^{j-1} z_2^{j-1}$ for some $1 \le j \le |\mathcal{F}'|$. 
Let $P_j = v_1 v_2 v_3 v_{4}$ and $P_{j+1} = v'_1 v'_2 v'_3  v'_{4}$.
Set 
\begin{align*}
W_j = \left( V ( \mathcal{F}' )  \cup  \bigcup_{1 \le j' < j} \{ z_1^{j'}, z_2^{j'} \} \right) \setminus \{v_3,v_4,v'_1,v'_2\}.
\end{align*}
Note that 
\begin{align*}
|W_j| = 4 | \mathcal{F}'| +  2  (j-1)  -4 < 6 |\mathcal{F}'|  \le 2 \eps n /3.
\end{align*}
Define $G_j = G [ V(G) \setminus W_j ]$, so $\delta^c_1(G_j) \ge 2n/3 + \eps n/3 \ge 2n/3+1$. 
Apply Lemma~\ref{lma:ifar} with $G = G_j$, $x_1 = v_{3}$, $x_2 = v_{4}$, $y_1 = v'_1$ and $y_2 = v'_2$, to obtain an edge $z_1^{j} z_2^{j}$ such that $ v_{3} v_{4} z_1^j z_2^j v'_1 v'_2$ is a properly coloured path in $G_j$.
This implies that $P_jz_1^j z_2^j P_{j+1}$ is a properly coloured path in~$G$.
Therefore, there exist vertex-disjoint edges $z_1^1 z_2^1, \dots, z_1^{|\mathcal{F}'|} z_2^{|\mathcal{F}'|}$ as desired.
Let $C$ be the properly coloured cycle obtained by concatenating $P_1, z_1^{1} z_2^{1} ,P_2, z_1^{2} z_2^{2}, \dots, P_{|\mathcal{F}'|}, z_1^{|\mathcal{F}'|} z_2^{|\mathcal{F}'|}$.
Note that $|C| = 6  | \mathcal{F}' | \le 2 \eps n /3$.

Suppose that $\mathcal{P}$ is a family of at most $ (8 \eps^2 n)/243$ vertex-disjoint properly coloured paths in $V(G) \setminus V(C)$.
Let $\mathcal{P}'$ be the family of vertex-disjoint properly coloured paths obtained from $\mathcal{P}$ by breaking up every path $P \in \mathcal{P}$ with $|P| \le 3$ into isolated vertices.
Hence, $\mathcal{P}'$ contains at most $ (8 \eps^2 n)/81$ paths and, for each path $P \in \mathcal{P}'$, either $|P| =1$ or $|P| \ge 4$.
Now, we assign each $P \in\mathcal{P}'$ to a path $Q  \in \mathcal{F'}$ such that $Q \in \mathcal{L}(V(P))$ if $|P| = 1$ and $Q \in \mathcal{L}(x_1,x_2;x_{\ell-1}, x_{\ell})$ if $P = x_1x_2 \dots x_{\ell}$ with $\ell \ge 4$.
Moreover, no two paths in $\mathcal{P}'$ are assigned to the same $Q \in \mathcal{F'}$.
Note that such an assignment exists by~\eqref{eqn:L3}.
Apply Proposition~\ref{prp:abspath} to each pair $(P, Q)$ and obtain a family $\mathcal{F}''$ of vertex-disjoint paths such that $|\mathcal{F}'| = |\mathcal{F}''|$.
Note that 
\begin{align*}
V(\mathcal{F}'') = V(\mathcal{F}')  \cup V(\mathcal{P}')= V(\mathcal{F}')  \cup V(\mathcal{P}).
\end{align*}
Moreover, there is a one-to-one correspondence between paths $P' \in \mathcal{F}'$ and $P'' \in \mathcal{F}''$ such that the endedges of $P'$ and $P''$ are the same.
Recall that $C$ is a properly coloured cycle containing $\mathcal{F}'$.
Let $C'$ be the cycle obtained from $C$ by replacing the paths in $\mathcal{F}'$ with paths in $\mathcal{F}''$.
Note that $C'$ is properly coloured and $V(C') = V(C) \cup V(\mathcal{P})$.
This completes the proof of Lemma~\ref{lma:abscycle}.
\end{proof}

\section{Proof of Theorem~$\ref{thm:PCHC2}$} \label{sec:proof}

First, we prove that $G$ contains a properly coloured triangle if $\delta^c_1(G) > 2n/3-1$.

\begin{prp} \label{prp:triangle}
Let $G$ be an edge-coloured graph on $n$ vertices with $\delta^c_1(G) > 2n/3-1$.
Then every vertex in $G$ is contained in a properly coloured triangle.
\end{prp}

\begin{proof}
Let $x$ be a vertex in $G$ and let $c$ be the edge-colouring on $G$.
Define $H$ to be the directed graph on $N(x)$ with directed edges $\overrightarrow{yz}$ if and only if $yz$ is an edge in $G[N(x)]$ with $c(yz) \ne c(xy) \ne c(xz)$.
For $y \in V(H)$, the outdegree of $y$ in $H$ is
\begin{align*}
	d^+_H(y)& \ge |\{z \in N_G(y) : c(yz) \ne c(xy)\}| +|\{z \in N_G(x) : c(xz) \ne c(xy)\}| - (n-2) \\
	& \ge  2 \delta^c_1(G) - n +2.
\end{align*}
Note that if $yz$ is a directed 2-cycle in $H$, then $\{x,y,z\}$ forms a properly coloured triangle in~$G$.
Hence, we may assume that $H$ does not contain any directed 2-cycles.
Moreover, we may assume that if $\overrightarrow{zy}$ is in $H$, then $c(zy) = c(xy)$.
By an averaging argument, there exists a vertex $y \in V(H)$ such that $d^-_H(y) \ge d^+_H(y) \ge 2 \delta^c_1(G) - n + 2$.
Therefore
\begin{align*}
	d_G(y) & \ge |\{z \in N_G(y) : c(yz) = c(xy)\}| + |\{z \in N_G(y) : c(yz) \ne c(xy)\}|\\
	& \ge d^-_H(y) + \delta^c_1(G) \ge 3 \delta^c_1(G) - n + 2.
\end{align*}
Since $d_G(y) \le n-1$, the inequality above implies that $\delta^c_1(G) \le 2n/3 - 1$, a contradiction.
\end{proof}

Finally, we prove Theorem~\ref{thm:PCHC2}.

\begin{proof}[Proof of Theorem~$\ref{thm:PCHC2}$]
Without loss of generality, we may assume that $\eps < 2^{-9}3^{-1}$.
Choose $n_0 \in \mathbb{N}$ large so that $\eps(1-2\eps /3 )n_0/3 \ge 1$, $n_0 \ge 243 \eps^{-3}$, and Lemma~\ref{lma:abscycle} holds.
Let $G$ be an edge-coloured graph on $n$ vertices with $\delta^c_1(G) \ge (2/3 + \eps) n$ as stated in Theorem~\ref{thm:PCHC2}.
By Proposition~\ref{prp:triangle}, $G$ contains a properly coloured triangle.

Suppose that $\ell$ is an integer with $4 \le \ell \le 2\eps n /3 $.
Since $\delta_1^c(G) \ge (2/3+ \eps ) n$, we can greedily construct a properly coloured path of length $2n/3$.
In particular, $G$ contains a properly coloured path $P = x_1 x_2 \dots x_{\ell - 2}$ of length~$\ell-3$.
Let $G'$ be the subgraph of $G$ obtained by removing all the vertices $x_3,x_4, \dots, x_{\ell-4}$.
Note that 
\begin{align*}
\delta_1^c(G') \ge \delta_1^c(G) - (\ell -6) \ge (2/3 + \eps/3)n \ge 2|G'|/3+1.
\end{align*}
Hence, Lemma~\ref{lma:ifar} implies that there exist an edge $z_1z_2$ in $G'$ such that $x_{\ell-3} x_{\ell-2} z_1 z_2 x_1 x_2$ is a properly coloured path.
Therefore, $x_1 x_2 \dots x_{\ell-2} z_1 z_2 x_1$ is a properly coloured cycle of length~$\ell$.

Suppose that $2\eps n / 3 < \ell  \le n$.
Let $C$ be the properly coloured cycle given by Lemma~\ref{lma:abscycle}, so $|C| \le 2 \eps n /3$.
Let $G''$ be the subgraph of $G$ obtained after removing all the vertices of $C$.
Note that 
\begin{align*}
\delta_1^c(G'') > \delta_1^c(G) - |C| \ge (2+\eps)n/3 \ge (2+\eps)|G''|/3.
\end{align*}
Note that $\eps |G''|/3 \ge \eps (1-2\eps /3 )n/3 \ge 1$.
By Lemma~\ref{lma:2factor1}, $G''$ can be covered by at most $ \lfloor 6 \eps^{-1} \rfloor$ vertex-disjoint properly coloured paths.
That is, we can find vertex-disjoint properly coloured paths $P_1, \dots, P_k$ in $G''$ such that $k \le 6 \eps^{-1} \le 8 \eps^2 n/243$, $V(P_i) \cap V(P_j) = \emptyset$ for all $i \ne j$ and $\bigcup_{1 \le i \le k } V(P_i) = V(G'')$.
By removing vertices in the paths, we may assume that the paths $P_1, \dots, P_k$ span exactly $\ell - |C|$ vertices.
By the property of~$C$ guaranteed by Lemma~\ref{lma:abscycle}, there exists a properly coloured cycle $C'$ with $V(C') = V(C) \cup \bigcup_{1 \le i \le k } V(P_i)$. 
Note that $|C'| = |C| + \sum_{1 \le i \le k } |P_i|  = \ell$.
Therefore, $C'$ is a properly coloured cycle of length~$\ell$ as required.
\end{proof}

\section{Acknowledgements}

The author would like to thank Katherine Staden for her valuable comments.
The author thanks the referees for their careful reading, one of whom in particular read the paper extremely carefully and suggested helpful clarifications.

\end{document}